\documentclass{article}

\usepackage[preprint]{neurips_2026}

\RequirePackage{amsthm,amsmath,amsfonts,amssymb}
\RequirePackage{graphicx}
\usepackage{url}
\usepackage[T1]{fontenc}
\usepackage{float}
\usepackage{babel}
\usepackage{amsmath}
\usepackage{amsfonts}
\usepackage{bbm}
\usepackage{algorithm}
\usepackage{multicol}
\usepackage[usenames,dvipsnames]{xcolor}
\usepackage[nottoc]{tocbibind}
\usepackage{comment}
\RequirePackage{hypernat}
\usepackage{sidecap}

\usepackage{comment}

\allowdisplaybreaks

\theoremstyle{plain}

\newtheorem{theorem}{Theorem}[section]
\newtheorem{lemma}[theorem]{Lemma}
\newtheorem{proposition}[theorem]{Proposition}
\newtheorem{corollary}[theorem]{Corollary}
\newtheorem{remark}[theorem]{Remark}
\newtheorem{assumption}[theorem]{Assumption}

\theoremstyle{remark}

\def \bbeta{\boldsymbol{\beta}}
\def \bx{\mathbf{x}}

\def \bZ{\mathbf{Z}}
\def \bbeta{\boldsymbol{\beta}}

\def \Xm {\mathbb{X}_{\rm pool}}
\def \tr {\mathrm{Tr}}

\def \bX{\mathbf{X}}

\def \be{\begin{align*}}
\def \ee{\end{align*}}

\def \I{\mathcal{I}}

\def \E{\mathbb{E}}
\def \P{\mathbb{P}}

\def \bmu{\pmb{\mu}}

\def \bZ {{\boldsymbol{Z}}}
\def \bX{\boldsymbol{X}}
\def \I {\mathbb{I}}

\newcommand{\op}{\mathrm{op}}

\title{Pooling Versus Ensembling for Ridge Regression Under Covariate Shift}
\usepackage{times}

\author{%
  Maya~Ramchandran\\
  Simons Institute for the Theory of Computing\\
  UC Berkeley\\
  \texttt{maya.ramchandran@berkeley.edu} \\
  \And
  Rajarshi Mukherjee \\
  Department of Biostatistics\\
  Harvard T.H. Chan School of Public Health \\
  \texttt{ram521@mail.harvard.edu} \\
}

\begin{document}

\maketitle

\begin{abstract}

Datasets in many settings naturally partition into clusters arising from sub-populations, batch effects, or aggregation across multiple sources. A common response to such heterogeneity is to ensemble learners trained on each cluster rather than fit a single model to the pooled data. Prior work motivating such approaches has typically considered settings in which both the covariate distribution and the conditional outcome model differ across clusters; the role of cluster-aware partitioning and ensembling based solely on the covariate distribution remains to be explored. We address this case for ridge-regularized least-squares regression under a linear outcome model and consider all ridge penalty values $\lambda \geq 0$, including the special case of the ridgeless predictor at $\lambda = 0$. By considering both fixed-effects and random-effects models, we argue that under random effects, an optimally tuned pooled ridge predictor always outperforms ensembles of individually optimally tuned predictors. For fixed effects, we derive a general formula for the pooled and ensembled predictors to characterize the role of both regression coefficients as well as the predictor distribution shifts. %We prove that when $\lambda > 0$, the predictor trained on the pooled data has uniformly smaller asymptotic prediction risk than the per-cluster ensemble for every value of the ridge penalty and across variations in cluster geometry. When $\lambda = 0$, this pattern depends on whether each of the base learners is in the under-parameterized regime, interpolation peak, or over-parameterized regime.  Simulations confirm the general dominance of the pooled learner across regimes, and reveal a second finding: cluster structure is essentially risk-neutral for both learners, with the gap between the two strategies governed by the ratio of the number of covariates to the number of samples rather than by cluster geometry. 
Together, these results generalize prior risk analyses of bagging and random-partition estimation using ridge and ridgeless regression predictors from the i.i.d. setting to encompass covariate shift and heterogeneity-aware partition structure.
\end{abstract}

\section{Introduction}
\label{sec:intro}
\par Heterogeneity in the distribution of covariates is a common artifact of modern datasets, where data can be comprised of natural sub-populations, clusters, batch effects, and multiple sources \citep{goh2017batch, chauhan2010data}. A long-standing question in this setting is whether prediction models should explicitly acknowledge this structure or ignore it. This question has a substantial classical literature on generalization across multiple distributions \citep{crammer2008learning, bendavid2010theory, mansour2009domain} and on pooled estimation under heterogeneity \citep{meinshausen2015maximin, li2022transfer}. One natural response has been the development of ensembling frameworks, in which the training data is first separated into its component clusters or natural partitions, a learning algorithm is trained independently on each cluster, and the resulting single-cluster predictors are combined using weights that reward cross-cluster generalization on the training set \citep{Patil2018, deodhar2007, Trivedi2015, Ramchandran2020, ramchandran2021ensembling,Guan2019a}. Such ensembles have been shown empirically to generalize more reliably than a single model trained on the pooled data across a range of base learners, including neural networks, random forests, and regularized regression. 

Throughout the paper, we refer to this strategy simply as \emph{ensembling}, and contrast it with \emph{pooling}, which refers to training a single model on the entire training set with the same base learner. More specifically, we compare two procedures built from the same base learner. The \emph{pooled} procedure trains a single predictor on the pooled training data. The \emph{ensemble} procedure trains one predictor separately on each known cluster and then averages, or more generally convexly combines, their predictions. Our goal is to determine whether the second strategy offers any advantage when clusters differ only in their covariate distributions and share the same conditional outcome model.

\par Although ensembling has been motivated by settings in which both the covariate distribution and the conditional outcome model vary across sources, it is far less clear what role cluster-aware partitioning plays when the heterogeneity is purely in the covariates. In this paper, we isolate this question by studying a setting with generalized covariate shift but no concept shift. Specifically, we analyze studying ridge-regularized least-squares regression with outcomes generated by a well-specified linear model whose coefficients are shared across clusters and test points; only the marginal distribution of the covariates varies both within the training set and between train and test sets.

%In this setting we show that for ridge regression where the regularization parameter $\lambda > 0$, merging uniformly dominates ensembling across every ridge penalty and every cluster geometry. For ridgeless regression at $\lambda = 0$, the relationship is dependent on the interplay between the regimes of the \emph{Pooled} and \emph{Ensemble} learners, but when both learners are in the under- or over-parameterized regimes, merging again dominates ensembling. More surprisingly, we find that cluster geometry is essentially risk-neutral for both approaches; the ensembling-versus-merging gap is governed by the aspect ratio (the ratio of the number of covariates to the number of samples), not by the angles, norms, or covariances of the cluster means.

\subsection{Summary of Contributions}

\paragraph{Theoretical results.} we provide general formulas for the risk of the ensemble and pooled learners using ridge-regularized least-squares regression as the base learner under a well-specified linear outcome model and arbitrary cluster-specific covariate distributions. When the ridge penalty $\lambda > 0$, we prove that the pooled predictor has uniformly smaller (asymptotic) prediction risk than the ensemble built on the true cluster partitions. This holds independently of the angles, norms, and covariance structures of the cluster means. When the ridge penalty $\lambda = 0$ for the ridgeless predictor, we show that if both learners are in either the underparameterized or overparameterized regime, the pooled predictor similarly outperforms the ensemble; the only point at which this relationship switches is when the pooled predictor is at the interpolation threshold where the total number of samples equals the number of covariates, whereas the base learners comprising the ensemble are in the over-parameterized regime. In this case, the implicit regularization caused by overparameterization allows the ensemble to achieve lower variance than the pooled predictor.

\paragraph{Simulation studies across regimes.} We confirm our theoretical results with simulations for both ridge and ridgeless predictors that span the underparameterized regime, the interpolation threshold, and the overparameterized regime. In every configuration we consider - varying cluster mean angles, cluster mean norms, and cluster covariance structures - the pooled predictor matches or outperforms the ensemble when $\lambda > 0$ (ridge regression). For ridgeless regression at $\lambda = 0$, the simulations confirm the theoretical result that the relationship between approaches depends on the regime, with the pooled learner largely dominating the ensemble when both are under- or over-parameterized. Furthermore, we find that the ensembling-vs-pooling gap is driven by the aspect ratio $p/n$, as opposed to the covariate distribution or partition strategy.

\subsection{Related Work}
\label{sec:related-work}

Our work builds on two closely related lines of recent research. \citet{hastie2022surprises} provide the exact asymptotic risk formulas for ridge and ridgeless regression that underpin the qualitative behavior we observe, and \citet{patil2023bagging} analyze a class of bagging procedures (subagging and splagging) in a supervised setting similar to ours. We defer discussion of the related unsupervised stacked-SVD setting of \citet{baharav2025stackedsvd} to Appendix~\ref{app:related-work}, since our focus is on supervised prediction.

\paragraph{Asymptotic risk of ridge and ridgeless regression.}
A line of recent work has produced precise asymptotic risk characterizations for high-dimensional regression in the proportional regime $p/n \to \gamma$, building on classical random matrix theory. \citet{dobriban2018high} provide early asymptotic risk formulas for ridge regression and classification in this regime. \citet{hastie2022surprises} subsequently derive exact formulas for the prediction risk of the minimum-norm (ridgeless) least-squares estimator and of ridge regression under a well-specified linear model with arbitrary covariate covariance. Their formulas make explicit several phenomena central to our experimental story: the ridgeless risk diverges at the interpolation threshold $\gamma = 1$ and decreases beyond it (the ``double descent'' shape, identified empirically by \citet{belkin2019reconciling} and analyzed as \emph{benign overfitting} by \citet{bartlett2020benign}), and even a small ridge penalty smooths the interpolation peak away. We use these results as a direct point of reference when interpreting the risk profile of our two learners.

\paragraph{Bagging in overparameterized regression.}
\citet{patil2023bagging} characterize the asymptotic prediction risk of two bagging variants for ridge and ridgeless least-squares predictors in the proportional-asymptotic regime: \emph{subagging}, in which $K$ predictors are trained on possibly overlapping random subsamples, and \emph{splagging} (split-aggregation), in which the data is partitioned via uniform random permutation into $K$ disjoint equal-size subsets and the per-partition predictors are averaged. Splagging is structurally similar to our ensemble learner with $K = 2$, with two key differences: their partitions are uniformly random, and they assume i.i.d.\ train and test data with no covariate shift or internal heterogeneity. Their main results are established under a well-specified linear model with arbitrary covariate covariance and arbitrary true coefficient vector. They show that properly tuned ridge regression is optimal among the methods considered, with no ensembling strategy surpassing it. One notable finding is that under isotropic covariates, the optimally-tuned subagged ridgeless predictor matches the asymptotic risk of optimally-tuned ridge; that is, bagging implicitly performs the regularization that ridge performs explicitly. This is further illustrated by comparing ensembles of ridgeless predictors trained on subsets to a single ridgeless predictor trained on the full dataset: the ensemble shows benefit largely at the full predictor's interpolation threshold, where each ensemble component has been pushed into the overparameterized regime and exhibits variance stabilization.

\citet{lejeune2020implicit} make this implicit-regularization view explicit, showing that in the proportional asymptotic regime, ensembles of ordinary least-squares predictors trained on random subsets behave like a single predictor with an effective ridge penalty determined by the subsample size. In a complementary fixed-dimension kernel-ridge setting, \citet{zhang2015divide} analyze a divide-and-conquer estimator that partitions the data into $K$ disjoint subsets, fits kernel ridge regression on each, and averages the predictors; they show that for $K$ not too large, this matches the convergence rate of the predictor trained on the full dataset. Together, these results show that bagging and partitioning act as implicit regularization, with primary gains for linear predictors concentrated in regimes where interpolation would otherwise produce high variance; none, however, exceed optimally tuned ridge. Our work extends this picture by asking how covariate heterogeneity and heterogeneity-aware partitioning affect the benefit of ensembling versus pooling: rather than assuming i.i.d.\ data with uniformly random partitions, we study a mixture of clusters with heterogeneous covariate distributions and partition along the underlying cluster structure, investigating whether the conclusions of \citet{patil2023bagging} continue to hold.

\section{Main Results}
We organize the main results into subsections covering the mathematical setup, the analyses of the ensemble and pooled learners, and their comparison.

\subsection{Setup}\label{sec:setup}
We theoretically compare the asymptotic risk of an ensemble predictor to that of a pooled predictor under a linear outcome model. The data is comprised of $K$ clusters (or sources): $\mathcal{D}_\ell=(y_{i,\ell},\bX_{i,\ell})_{i=1}^{n_\ell}$, $\ell=1,\ldots,K$, with outcome $y\in\mathbb{R}$ and covariates $\bX\in\mathbb{R}^p$. Let $\hat{f}_\ell$ be a predictor fit on each cluster $\mathcal{D}_\ell$, and let $\hat{f}_{\rm pool}$ be a predictor fit on the pooled dataset $\mathcal{D}=\{\mathcal{D}_\ell:\ell=1,\ldots,K\}$. Our main results compare the prediction error weighted combinations of the $\hat{f}_\ell$'s to that of $\hat{f}_{\rm pool}$, with all learners fit by ridge-regularized least squares.

To stage the analysis, we define
\[
\mathbf{Y}_\ell=(Y_{1,\ell},\ldots,Y_{n_\ell,\ell})^\top\in\mathbb{R}^{n_\ell},
\qquad
\mathbb{X}_\ell=(\bX_{1,\ell},\ldots,\bX_{n_\ell,\ell})^\top\in\mathbb{R}^{n_\ell\times p}
\]
for the outcome vector and design matrix in cluster $\ell$, and let $\mathbf{Y}_{\rm pool}=(\mathbf{Y}_1^\top,\ldots,\mathbf{Y}_K^\top)^\top$ and $\mathbb{X}_{\rm pool}=(\mathbb{X}_1^\top,\ldots,\mathbb{X}_K^\top)^\top$ denote their pooled counterparts. The ridge predictors are then
\begin{align*}
\hat{f}_\ell(\bx)
 &=\hat{\boldsymbol{\beta}}_{\lambda_\ell}^\top\bx,
&\hat{\boldsymbol{\beta}}_{\lambda_\ell}
 &=(\mathbb{X}_\ell^\top\mathbb{X}_\ell+\lambda_\ell\mathbb{I})^{-1}\mathbb{X}_\ell^\top\mathbf{Y}_\ell,
&\lambda_\ell&>0;\\
\hat{f}_{\rm pool}(\bx)
 &=\hat{\boldsymbol{\beta}}_{\lambda_0}^\top\bx,
&\hat{\boldsymbol{\beta}}_{\lambda_0}
 &=(\mathbb{X}_{\rm pool}^\top\mathbb{X}_{\rm pool}+\lambda_0\mathbb{I})^{-1}\mathbb{X}_{\rm pool}^\top\mathbf{Y}_{\rm pool},
&\lambda_0&>0,
\end{align*}
with ridgeless ($\lambda\downarrow 0$) counterparts defined via the Moore-Penrose pseudoinverse:
\begin{align*}
\hat{f}_{\ell,+}(\bx)
 &=\hat{\boldsymbol{\beta}}_{+,\ell}^\top\bx,
&\hat{\boldsymbol{\beta}}_{+,\ell}
 &=(\mathbb{X}_\ell^\top\mathbb{X}_\ell)^+\mathbb{X}_\ell^\top\mathbf{Y}_\ell;\\
\hat{f}_{\rm pool,+}(\bx)
 &=\hat{\boldsymbol{\beta}}_{+}^\top\bx,
&\hat{\boldsymbol{\beta}}_{+}
 &=(\mathbb{X}_{\rm pool}^\top\mathbb{X}_{\rm pool})^+\mathbb{X}_{\rm pool}^\top\mathbf{Y}_{\rm pool}.
\end{align*}
We work under the following assumptions.

\begin{assumption}\label{assumptions} The following hold for $\ell=1,\ldots,K$.
\begin{enumerate}
\item[(i)] \textbf{Covariate Distribution and Mean Shifts:} $\bX_{i,\ell}=\bmu_\ell+\Sigma^{1/2}\mathbf{Z}_{i,\ell}$, where $\mathbf{Z}_{i,\ell}\in\mathbb{R}^p$ has i.i.d.\ mean-zero subgaussian coordinates and $\Sigma$ is positive definite with $\eta\le\lambda_1(\Sigma)\le\lambda_p(\Sigma)<\eta^{-1}$ for some fixed $\eta>0$. Moreover, $K$ is fixed and the cluster means $\bmu_1,\ldots,\bmu_K$ have uniformly bounded Euclidean norms.

\item[(ii)] \textbf{Outcome Regression:} $Y_{i,\ell}=\boldsymbol{\beta}^\top\bX_{i,\ell}+\varepsilon_{i,\ell}$, with $\varepsilon_{i,\ell}\perp\bX_{i,\ell}$ subgaussian, mean $0$, and variance $\sigma_\varepsilon^2$.

\item[(iii)] \textbf{Test Data:} $y^\star=\bbeta^\top\bx^\star+\varepsilon^\star$ with $\E(\bx^\star)=0$, $\mathrm{Var}(\bx^\star)=\mathbb{I}$, $\varepsilon^\star\perp\bx^\star$, $\E(\varepsilon^\star)=0$, $\mathrm{Var}(\varepsilon^\star)=\sigma_\star^2$.

\item[(iv)] \textbf{Asymptotic Regime:} For $\ell\ge 0$, $p/n_\ell\to\gamma_\ell\in[0,\infty)$ and $\lambda_\ell=\rho_\ell n_\ell$ for fixed $\rho_\ell\in(0,\infty)$, where $n_0=n=\sum_{\ell=1}^K n_\ell$.
\end{enumerate}
\end{assumption}

\begin{remark}
Assumption~\ref{assumptions} focuses on mean shifts as the source of covariate heterogeneity. This choice keeps the comparison between pooled and ensemble learners analytically transparent while still capturing a central form of between-cluster variation. The same random matrix arguments can be extended to more general forms of covariate shift, though at the cost of heavier notation.
 Our choice of a centered, isotropic test data point $\bx^\star$ is also driven by the same goal of keeping our theoretical results relatively concise. Finally, the cluster labels are treated as known. This is natural when the data come from distinct studies or populations. When clusters are estimated from the data, additional uncertainty enters the analysis; we leave this extension for future work.
\end{remark}

Under these assumptions, we study the risk of ensemble versus pooled learners. For convex weights $\mathbf{w}=(w_1,\ldots,w_K)\in\mathbb{R}_+^K$ with $\sum_{\ell=1}^K w_\ell=1$ and $\boldsymbol{\rho}=(\rho_1,\ldots,\rho_K)$ with $\lambda_\ell=\rho_\ell n_\ell$, define the ensemble estimators
\[
\widehat\bbeta_{\mathrm{ens},\boldsymbol{\rho}}(\mathbf{w})=\sum_{\ell=1}^K w_\ell\widehat\bbeta_{\lambda_\ell},
\qquad
\widehat\bbeta_{\mathrm{ens},+}(\mathbf{w})=\sum_{\ell=1}^K w_\ell\widehat\bbeta_{+,\ell},
\qquad w_\ell\ge 0,\ \sum_{\ell=1}^K w_\ell=1.
\]
Conditional on $\{\mathbb{X}_\ell\}_{\ell=1}^K$, their mean squared prediction errors are
\begin{align*}
R_{\mathrm{ens},\boldsymbol{\rho}}(\mathbf{w})
 &=\E\!\left[(y^\star-\bx^{\star\top}\widehat\bbeta_{\mathrm{ens},\boldsymbol{\rho}}(\mathbf{w}))^2\right],
&R_{\mathrm{ens},+}(\mathbf{w})
 &=\E\!\left[(y^\star-\bx^{\star\top}\widehat\bbeta_{\mathrm{ens},+}(\mathbf{w}))^2\right].
\end{align*}
The pooled risks are defined analogously: for $\lambda_0=\rho_0 n_0$,
\begin{align*}
R_{\mathrm{pool},\rho_0}
 &=\E\!\left[(y^\star-\bx^{\star\top}\widehat\bbeta_{\lambda_0})^2\right],
&R_{\mathrm{pool},+}
 &=\E\!\left[(y^\star-\bx^{\star\top}\widehat\bbeta_{\mathrm{pool},+})^2\right].
\end{align*}

The next two subsections analyze the ensemble and pooled methods in turn; we compare them in Section~\ref{sec:compare}. Stating the results requires further notation. Let $F_{\Sigma,p}$ denote the empirical spectral distribution of $\Sigma$, and for $\gamma,\rho>0$ let $m_\gamma(\rho)$ be the unique solution \citep{knowles2017anisotropic} of
\[
\frac{1}{m_\gamma(\rho)}=\rho+\gamma\int\frac{x}{1+xm_\gamma(\rho)}\,dF_{\Sigma,p}(x).
\]
Define $Q_\gamma(\rho)=[\rho\{\I+m_\gamma(\rho)\Sigma\}]^{-1}$, $a_\gamma(\rho)=\tfrac{1}{p}\tr Q_\gamma(\rho)$, $s_\gamma(\rho)=\tfrac{1}{p}\tr\{\Sigma Q_\gamma(\rho)\}$, and the variance functional
\[
V_\gamma(\rho)=\gamma\{a_\gamma(\rho)+\rho a_\gamma'(\rho)\},
\]
with the derivative taken with respect to $\rho$. Cluster-specific values are written $Q_\ell=Q_{\gamma_\ell}(\rho_\ell)$, $a_\ell=a_{\gamma_\ell}(\rho_\ell)$, $V_\ell=V_{\gamma_\ell}(\rho_\ell)$. Finally, for deterministic $U\in\mathbb{R}^{p\times r}$ of fixed rank $r$ and uniformly bounded operator norm, set
\[
\mathcal K_{\gamma,U}(\rho)
=Q_\gamma(\rho)
-Q_\gamma(\rho)U
 \bigl[\{1+\gamma s_\gamma(\rho)\}\I+U^\top Q_\gamma(\rho)U\bigr]^{-1}
 U^\top Q_\gamma(\rho),
\]
which when $U=\bmu$ is a single vector reduces to
\[
\mathcal K_{\gamma,\bmu}(\rho)
=Q_\gamma(\rho)
-\frac{Q_\gamma(\rho)\bmu\bmu^\top Q_\gamma(\rho)}
       {1+\gamma s_\gamma(\rho)+\bmu^\top Q_\gamma(\rho)\bmu}.
\]
We are now ready to state the main results.

\subsection{Analysis of Ridge Regression Ensembles}\label{sec:ensemble}
We split by the assumption on $\boldsymbol{\beta}$, treating the random- and fixed-effects cases in turn. The next theorem collects both.

\begin{theorem}\label{thm:gen_ensemble}
Assume \ref{assumptions}.
\begin{enumerate}
\item\label{thm:gen_ensemble_re} Let $\bbeta$ be random with $\E(\bbeta)=0$ and $\mathrm{Var}(\bbeta)=\tau^2\I/p$. Then
\[
\begin{aligned}
R_{\mathrm{ens},\boldsymbol{\rho}}(\mathbf{w})
&=\sigma_\star^2
+\sum_{\ell=1}^K w_\ell^2
\left\{\tau^2\rho_\ell^2[-a_{\gamma_\ell}'(\rho_\ell)]
       +\sigma_\varepsilon^2 V_{\gamma_\ell}(\rho_\ell)\right\}\\
&\quad+2\tau^2\sum_{1\le\ell<j\le K}
       w_\ell w_j\rho_\ell\rho_j\,
       \tfrac{1}{p}\tr(Q_\ell Q_j)
+o(1),\\
R_{\mathrm{ens},+}(\mathbf{w})
&=\lim_{\substack{\rho_\ell\downarrow 0\\ \ell=1,\ldots,K}}
   R_{\mathrm{ens},\boldsymbol{\rho}}(\mathbf{w})+o(1).
\end{aligned}
\]

\item\label{thm:gen_ensemble_fe} If $\limsup\|\bbeta\|<\infty$, set $\mathcal K_\ell(\rho)=\mathcal K_{\gamma_\ell,\bmu_\ell}(\rho)$ and $\psi_\ell(\rho)=\bbeta^\top\mathcal K_\ell(\rho)\bbeta$. Then
\[
\begin{aligned}
R_{\mathrm{ens},\boldsymbol{\rho}}(\mathbf{w})
&=\sigma_\star^2
+\sum_{\ell=1}^K w_\ell^2
\left\{\rho_\ell^2[-\psi_\ell'(\rho_\ell)]
       +\sigma_\varepsilon^2 V_{\gamma_\ell}(\rho_\ell)\right\}\\
&\quad+2\sum_{1\le\ell<j\le K}
       w_\ell w_j\rho_\ell\rho_j\,
       \bbeta^\top\mathcal K_\ell(\rho_\ell)\mathcal K_j(\rho_j)\bbeta
+o(1),\\
R_{\mathrm{ens},+}(\mathbf{w})
&=\lim_{\substack{\rho_\ell\downarrow 0\\ \ell=1,\ldots,K}}
   R_{\mathrm{ens},\boldsymbol{\rho}}(\mathbf{w})+o(1).
\end{aligned}
\]
\end{enumerate}
\end{theorem}

In the proof, the first summand in each risk follows fairly directly from \cite{hastie2019surprises}, with extra care for the mean shifts. The cross terms require more work, especially in the fixed-effects case: there we use a careful leave-two-out technique to control bilinear forms of resolvents of the sum of sample covariance matrices of non-mean-zero random vectors, where the vectors entering the bilinear form themselves depend on the resolvent's entries. Specializing to the optimally weighted ensemble requires the asymptotic behavior of the optimal weights, which we describe next.

\begin{theorem}\label{thm:ensemble_optimal}
Assume \ref{assumptions} and let $\bbeta$ be random with $\E(\bbeta)=0$ and $\mathrm{Var}(\bbeta)=\tau^2\I/p$. For each $\ell=1,\ldots,K$, let $\lambda_\ell^\star=n\rho_\ell^\star$ denote the optimal tuning for $\hat{f}_\ell(\bx)=\bx^\top\hat{\bbeta}_{\lambda_\ell}$, and let $\mathbf{w}^\star$ denote the optimal convex weights for ensembling the optimally tuned $\hat{f}_\ell$. Define $A^\star_{\mathrm{RE}}\in\mathbb{R}^{K\times K}$ by
\[
A^\star_{\mathrm{RE}}(\ell,\ell)
=\tau^2(\rho_\ell^\star)^2[-a'_{\gamma_\ell}(\rho_\ell^\star)]
+\sigma_\varepsilon^2 V_{\gamma_\ell}(\rho_\ell^\star),
\quad
A^\star_{\mathrm{RE}}(\ell,\ell')
=\tau^2\rho_\ell^\star\rho_{\ell'}^\star\,
 \tfrac{1}{p}\tr(Q_\ell^\star Q_{\ell'}^\star),
\]
where $Q_\ell^\star=Q_{\gamma_\ell}(\rho_\ell^\star)$. Then:
\begin{enumerate}
\item The risk of $\sum_{\ell=1}^K w_\ell\hat{\bbeta}_{\lambda_\ell^\star}$ equals $\sigma_\star^2+\mathbf{w}^\top A^\star_{\mathrm{RE}}\mathbf{w}+o(1)$.

\item If some $\widetilde{\mathbf{w}}\in\arg\min_{\mathbf{w}\in\mathbb{R}^K: \mathbf{1}^\top\mathbf{w}=1}\mathbf{w}^\top A^\star_{\mathrm{RE}}\mathbf{w}$ satisfies $\widetilde{\mathbf{w}}\in\mathbb{R}_+^K$, the risk of $\sum_{\ell=1}^K w_\ell^\star\hat{\bbeta}_{\lambda_\ell^\star}$ equals $\sigma_\star^2+\bigl(\mathbf{1}^\top(A^\star_{\mathrm{RE}})^{-1}\mathbf{1}\bigr)^{-1}+o(1)$.
\end{enumerate}
\end{theorem}

We state the optimally weighted, optimally tuned ensemble result only for the random-effects setting because this case yields a compact and interpretable expression. The same strategy can be applied in the fixed-effects setting, but the resulting optimal weights depend on the configuration of the mean shifts $\bmu_1,\ldots,\bmu_K$ and on their relationship with the true coefficient vector $\bbeta$. These relationships are challenging to estimate from data in the proportional asymptotic regime, which prevents a clean general comparison between pooled and ensemble learners under fixed effects. The same strategy as in Theorem~\ref{thm:ensemble_optimal} can nonetheless be applied to fixed-effects optimal ensembling, which we leave for future work.

\subsection{Analysis of Pooled Ridge Regression}\label{sec:pooled}
We split by the assumption on $\boldsymbol{\beta}$, treating the random- and fixed-effects cases in turn. The next theorem collects both.

\begin{theorem}\label{thm:Pooled}
Assume Assumption~\ref{assumptions} and $n_\ell/n\to\pi_\ell\in(0,1)$.
\begin{enumerate}
\item\label{thm:Pooled_re} Let $\bbeta$ be random with $\E(\bbeta)=0$ and $\mathrm{Var}(\bbeta)=\tau^2\I/p$.  Then
\begin{align*}
R_{\mathrm{pool},\rho_0}
&=\sigma_\star^2
+\tau^2\rho_0^2[-a_{\gamma_0}'(\rho_0)]
+\sigma_\varepsilon^2 V_{\gamma_0}(\rho_0)+o(1),\\
R_{\mathrm{pool},+}
&=\lim_{\rho_0\downarrow 0}R_{\mathrm{pool},\rho_0}+o(1).
\end{align*}

\item\label{thm:Pooled_fe} If $\limsup\|\bbeta\|<\infty$, set $U_0=[\sqrt{\pi_1}\bmu_1,\ldots,\sqrt{\pi_K}\bmu_K]$, $\mathcal K_0(\rho)=\mathcal K_{\gamma_0,U_0}(\rho)$, and $\psi_0(\rho)=\bbeta^\top\mathcal K_0(\rho)\bbeta$. Then
\begin{align*}
R_{\mathrm{pool},\rho_0}
&=\sigma_\star^2
+\tau^2\rho_0^2[-\psi_0'(\rho_0)]
+\sigma_\varepsilon^2 V_{\gamma_0}(\rho_0)+o(1),\\
R_{\mathrm{pool},+}
&=\lim_{\rho_0\downarrow 0}R_{\mathrm{pool},\rho_0}+o(1).
\end{align*}
\end{enumerate}
\end{theorem}

As in the ensemble analysis, most of the proof effort goes into the fixed-effects case, where we again must control bilinear forms of resolvents of the sum of sample covariance matrices of non-mean-zero random vectors, with the vectors entering the form depending on the resolvent's entries. With these results, we can now compare the two methods.

\subsection{Comparison of Ensemble vs Pooled Learners}\label{sec:compare}
We compare the pooled and ensemble predictors via specific instances of Theorems~\ref{thm:gen_ensemble} and~\ref{thm:Pooled}. Setting $K=2$, $\Sigma=\I$, and random effects on $\boldsymbol{\beta}$ admits a clean comparison and lets us quantify the precise benefit of the pooled learner over the optimally tuned, optimally weighted ensemble.

\begin{corollary}\label{cor:comp_k2_iso_re}
Assume \ref{assumptions} with $\Sigma=\I$ and let $\bbeta$ be random with $\E(\bbeta)=0$ and $\mathrm{Var}(\bbeta)=\tau^2\I/p$. T. Set $\gamma_0=p/n$, $\gamma_\ell=p/n_\ell$ for $\ell=1,\ldots,K$, $\alpha=\sigma_\varepsilon^2/\tau^2$, $\rho_\ell^\star=\gamma_\ell\alpha$ for $\ell\ge 0$, and
\[
\kappa_\ell:=\kappa(\gamma_\ell)
=\frac{-(\alpha-1+\gamma_\ell^{-1})+\sqrt{(\alpha-1+\gamma_\ell^{-1})^2+4\alpha}}{2},
\qquad \ell\ge 0.
\]
With $\boldsymbol{\kappa}=(\kappa_1,\ldots,\kappa_K)^\top$, set $B=\mathrm{diag}(\kappa_1(1-\kappa_1),\ldots,\kappa_K(1-\kappa_K))+\boldsymbol{\kappa}\boldsymbol{\kappa}^\top$. Then:
\begin{enumerate}
\item The optimally tuned pooled estimator is asymptotically no worse than the optimally tuned, optimally weighted ensemble if and only if

\[
    \kappa_0
    \le
    \min_{\mathbf{w}=(w_1,\ldots,w_k)\in\mathbb{R}_+^K:\sum\limits_{\ell=1}^Kw_{\ell}=1} \mathbf{w}^\top B\mathbf{w}.
\]

\item Suppose $K=2$, and let $R_{\rm ens}^\star$ and $R_{\rm pool}^\star$ denote the mean squared prediction errors of the optimally weighted, optimally tuned ensemble and of the optimally tuned pooled regression. Then
\[
\frac{R_{\rm ens}^\star-\sigma_\star^2}
     {R_{\rm pool}^\star-\sigma_\star^2}
=\frac{\kappa_1\kappa_2(1-\kappa_1\kappa_2)}
      {\kappa_0(\kappa_1+\kappa_2-2\kappa_1\kappa_2)}+o_\P(1)>1+o_{\P}(1).
\]

\item Suppose $K=2$ and $n_1=n_2$, and let $\gamma=p/n$ where $n=n_1+n_2=2n_1$. Let $R_{\rm ens,+}$ and $R_{\rm pool,+}$ denote the mean squared prediction errors of the optimally weighted ridgeless ensemble and of pooled ridgeless least squares. Then:
\begin{enumerate}
\item If $\gamma<\tfrac{1}{2}$, $\dfrac{R_{\rm ens,+}}{R_{\rm pool,+}}=\dfrac{1-\gamma}{1-2\gamma}+o(1)>1+o_\P(1)$.

\item If $\gamma>1$, $\dfrac{R_{\rm ens,+}}{R_{\rm pool,+}}>1+o_\P(1)$ if and only if
$\dfrac{\sigma_\varepsilon^2}{\tau^2}<\dfrac{(2\gamma+1)(2\gamma-1)(\gamma-1)}{4\gamma^2(3\gamma-1)}.$
\end{enumerate}
\end{enumerate}
\end{corollary}
%Appendix~\ref{app:ridgeless-validation} provides a validation figure for the ridgeless risk ratio in Corollary~\ref{cor:comp_k2_iso_re}\,(3a), showing agreement between the theoretical limit and simulations in the under-parametrized regime.
%\subsection{Additional ridgeless validation}
%\label{app:ridgeless-validation}
The Corollary above verifies that for random effects, the pooled learners perform better than ensembles -- both for optimally tuned learners and the ridgeless case. Figure~\ref{fig:F3} provides an additional validation of the limit in Corollary~\ref{cor:comp_k2_iso_re} part (3c) in the ridgeless setting. We consider the case $K=2$ with equal cluster sizes, as in the theoretical setup, and vary the aspect ratio $\gamma=p/n$ in the regime where both the pooled learner and the per-cluster base learners in the ensemble are underparameterized ($0 < \gamma < .5$). We find that the simulated MSE ratio closely follows the theoretical limit from Corollary~\ref{cor:comp_k2_iso_re} (3c).

\begin{SCfigure}[50][ht]
    \centering
    \includegraphics[width=0.6\textwidth]{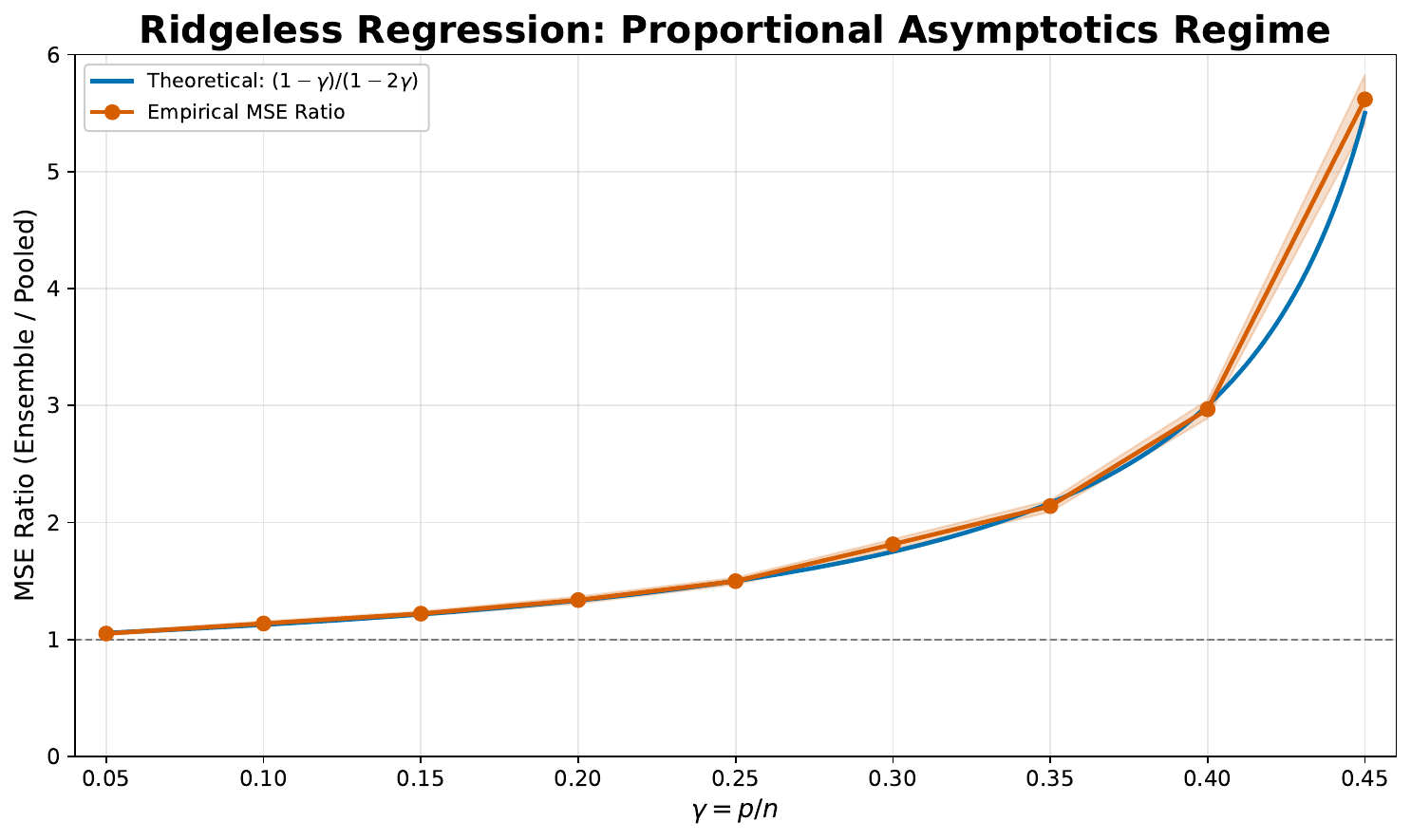}
    \caption{Ratio of \emph{Ensemble} MSE to \emph{Pooled} MSE for ridgeless learners with $K = 2$; total training sample size $n= 1000$, with $n_1 = n_2 = 500$, and $\gamma = p/n$ in the range $[0, p/n_1]$ so that both methods are underparameterized. For each cluster $\ell$, covariates are drawn from a $\mathbb{N}_p(\mu_{\ell}, \mathbb{I})$, with $\mu_1 = - \mu_2$. The theoretical limit from Corollary \ref{cor:comp_k2_iso_re} (3a) is shown in red, matching our simulated results across all values of $\gamma$.}
    \label{fig:F3}
\end{SCfigure}

\section{Simulations}

We conduct a comprehensive simulation study to validate our theoretical results. We describe the common setup, then present results for ridge and ridgeless estimators with a particular focus on the effects of cluster heterogeneity. All experiments were run using 4 CPU cores on a local machine.

\subsection{Simulation Setup}
\label{sec:sim-setup}

We simulate $n = 2000$ total training observations partitioned into $K = 2$ equal-sized clusters with $n_1 = n_2 = 1000$. Within each cluster $\ell \in \{1, 2\}$, the covariates are independent draws from a multivariate normal distribution with cluster-specific mean and identity covariance: $\mathbb{X}_{\ell} \sim \mathbb{N}_p(\mu_{\ell}, \mathbb{I})$ where $\mu_{\ell} \in \mathbb{R}^p$, 
so that all covariates are uncorrelated within and across clusters and have unit marginal variance. The first cluster's mean $\mu_1$ is a unit-norm random vector drawn from $\mathbb{N}_p(0, \mathbb{I})$ and rescaled so that $\|\mu_1\| = 1$; the second cluster's mean $\mu_2$ is constructed via Gram-Schmidt orthogonalization to have a prescribed angle with $\mu_1$ and a prescribed norm (the specific values used are described in Section~\ref{sec:sim-hetero}). The means $\mu_1$ and $\mu_2$ are redrawn independently in every repetition.

The coefficient vector $\boldsymbol{\beta} \in \mathbb{R}^p$ used to generate the outcome is drawn from $\mathrm{N}_{p}(0,\mathbb{I})$ and normalized to unit norm ($\|\boldsymbol{\beta}\| = 1$); the same $\boldsymbol{\beta}$ is used to generate training and test outcomes within each repetition. For cluster $\ell \in \{1, 2\}$ and samples $i=1,\ldots,n_\ell$, training data is generated as $\mathbb{X}_{\ell} \sim \mathrm{N}_{p}(\mu_{\ell}, \mathbb{I})$ with response $Y_{\ell}=\boldsymbol{\beta}^{\top}\bX_{\ell}+\varepsilon_{\ell}$, where $\varepsilon_{\ell} \sim \mathrm{N}_{n_\ell}(0, \sigma^2\mathbb{I})$ and $\sigma = 1/\sqrt{5}$, giving a signal-to-noise ratio $\mathrm{SNR} = \|\boldsymbol{\beta}\|^2 / \sigma^2 = 5$ (matching the highest SNR considered by \citep{hastie2022surprises}). 

The test data $\bX^{\star} \in \mathbb{R}^{n_{\text{test}} \times p}$ is drawn from a standard multivariate normal; $\bX^{\star} \sim \mathrm{N}_{p}(0, \mathbb{I})$, with $n_{\text{test}} = 1000$. The outcome is simulated with no noise, $Y^{\star} = \bX^{\star} \boldsymbol{\beta}$, isolating the prediction risk of each estimator on a fixed reference distribution and matching the theoretical setup of \citep{hastie2022surprises}. The covariate dimension is varied across $p \in \{250, 500, 750, 1050, 1500, 2050, 3000, 4000, 5000, 6000\}$, spanning the underparameterized regime, the cluster interpolation threshold at $p = n/K = 1000$, the Pooled interpolation threshold at $p = n = 2000$, and the deep overparameterized regime. We avoid values of $p$ that exactly coincide with these thresholds to prevent numerical instability at the singular boundary. All results are averaged over 100 independent repetitions, with 95\% confidence intervals computed from the $t$-distribution.

\paragraph{Base Learners}
We consider ridge and ridgeless base learners as defined in Section \ref{sec:setup}. We compute the ridgeless solution numerically using the SVD-based routine \texttt{numpy.linalg.lstsq} with the default rank-cutoff tolerance, which is mathematically equivalent to the pseudoinverse solution but more numerically stable. For the ridge solution, the regularization parameter $\lambda$ is selected by leave-one-out cross-validation over a logarithmic grid of 20 values spanning $[10^{-6}, 10^{6}]$ using the \texttt{RidgeCV} implementation in \texttt{scikit-learn} \citep{pedregosa2011scikit}; $\lambda$ is tuned independently for every learner.

\subsection{Parameter Configurations}
\label{sec:sim-hetero}

To assess how cluster heterogeneity modulates the ensemble--Pooled comparison, we vary two geometric parameters of the cluster means while holding the rest of the setup fixed:
\begin{itemize}
    \item \textbf{Fixed angle, varying norm.} We fix the angle between $\mu_1$ and $\mu_2$ at $60^\circ$ and set $\|\mu_1\| = 1$, $\|\mu_2\| \in \{1, 5, 10\}$. This sweeps the relative signal strength of the two clusters from balanced to highly asymmetric.
    \item \textbf{Fixed norm, varying angle.} We fix $\|\mu_1\| = \|\mu_2\| = 1$ and set the angle between them to $\{30^\circ, 60^\circ, 90^\circ\}$. This sweeps the directional divergence between the two clusters from highly aligned to orthogonal.
\end{itemize}

We evaluate each parameter configuration for all 10 values of $p$ over 100 repetitions; we plot the MSE ratio of the ensemble to the pooled learner in Figure \ref{fig:F1} (a) and (b), and the corresponding raw MSE values for each method in Figure \ref{fig:F1} (c) and (d). We performed additional experiments varying the signal-to-noise in the range $\{1, 10\}$, as well as considering fixed, non-zero $\boldsymbol{\beta}$ vectors for generating the outcome, and found that the results did not change; for simplicity, we show only the results corresponding to the specific parameter configurations discussed above.

\begin{figure}[t!]
  \centering
  \includegraphics[width=\linewidth]{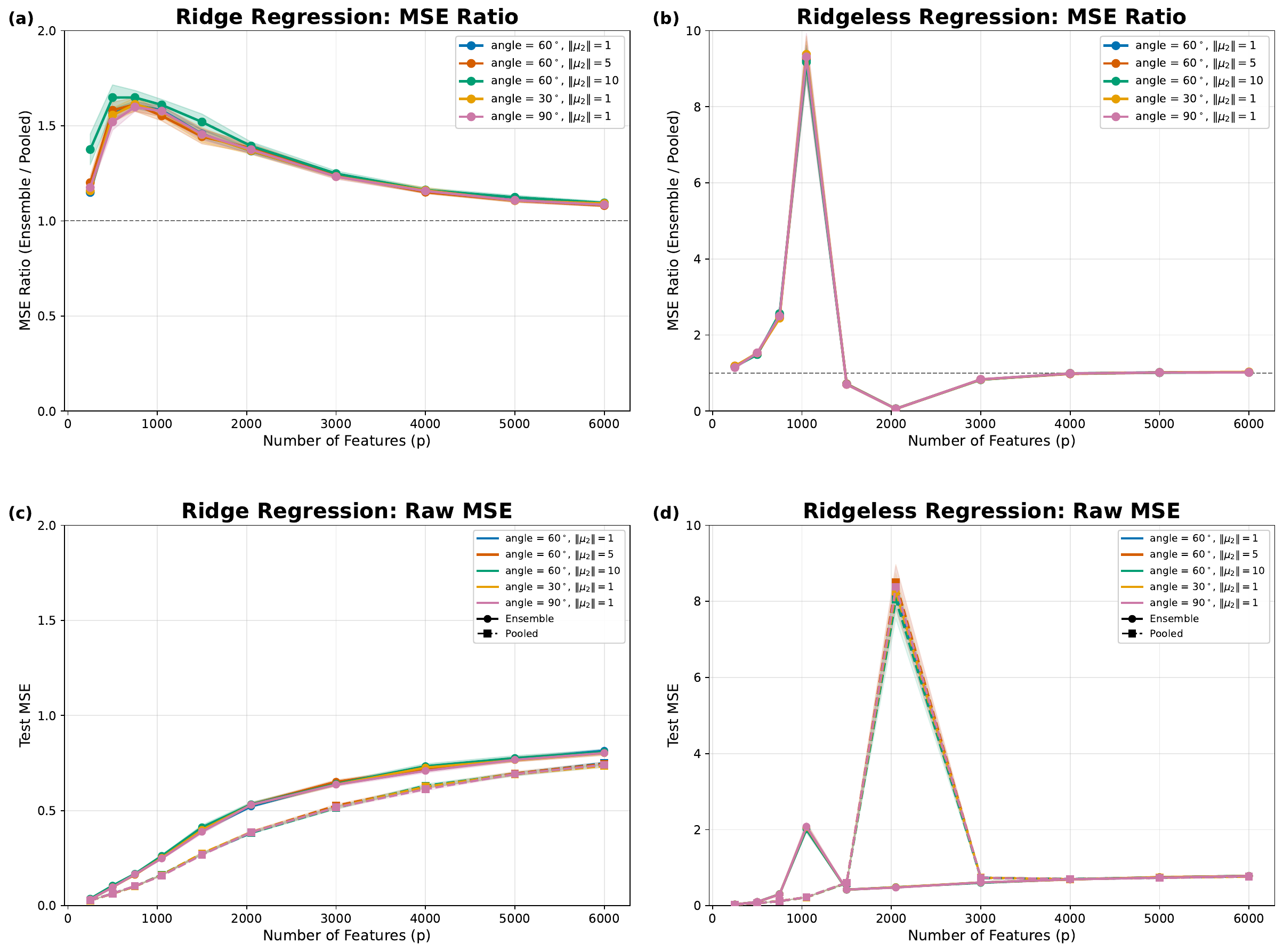}
  \caption{MSE ratio (\emph{Ensemble} / \emph{Pooled}) in panels (a, b) and raw ensemble and pooled MSEs in panels (c, d); ridge regression in (a, c) and ridgeless regression in (b, d). Training data are drawn from two multivariate normal clusters with means $\mu_1, \mu_2$ and identity covariance under a linear outcome model, with $n_1 = n_2 = 1000$ ($n = 2000$); test points are $n_{\text{test}} = 1000$ draws from a standard multivariate normal. Mean results and 95\% CIs are plotted over 100 replicates per configuration. Each panel overlays two scenarios: (i)~angle between $\mu_1, \mu_2$ fixed at $60^{\circ}$, $\lVert\mu_1\rVert = 1$, varying $\lVert\mu_2\rVert$; (ii)~$\lVert\mu_1\rVert = \lVert\mu_2\rVert = 1$, varying angle. For ridge learners, \textit{Pooled} uniformly outperforms \textit{Ensemble}. For ridgeless learners, behavior is governed by two interpolation peaks (\textit{Ensemble} at $p \approx n_1$, \textit{Pooled} at $p \approx n$); away from these peaks, the two are nearly equivalent. Varying the angle or norms has a negligible effect throughout.}
  \label{fig:F1}
\end{figure}

\section{Discussion}
This paper presents a precise asymptotic comparison between pooled and ensemble learners fit with ridge-regularized linear regression in a proportional asymptotic regime. Our results demonstrate the benefits of pooled methods in random-effects settings and quantify the precise efficiency gains over ensemble learners.

The simulations complement the theory and additionally capture behavior near the interpolation thresholds for ridgeless regression. For the ridgeless estimator, the MSE ratio shows distinct patterns across three regimes. When both the ensemble base learners and the pooled learner are underparameterized ($p<n_\ell$), the advantage of the pooled approach grows with $p$, vanishing only in the fixed-dimensional limit $p\ll n_\ell$ where the two methods perform comparably. At the ensemble's interpolation threshold ($p\approx n_\ell$), its MSE spikes and drives the ratio sharply up. As $p$ moves into the intermediate regime $n_\ell<p<n$, the ratio dips slightly below $1$ and continues to drop as $p$ approaches the pooled learner's interpolation threshold ($p\approx n$); at this point each ensemble base learner is already overparameterized, and the ensemble enjoys its largest advantage; this complements results from \cite{patil2023bagging}. Notably, the MSE spike of the ensemble is roughly a quarter the size of that for the pooled approach, indicating that the former's worst-case performance is considerably better than the latter's. Once both methods are overparameterized ($p>n$), the ratio approaches $1$ from below as both MSEs converge to the null risk. The picture differs sharply for ridge regression with $\lambda$ tuned optimally per learner: no regime favors the ensemble. The MSE ratio stays between $1.06$ and $1.62$ across all $p$, and the raw MSEs are smooth and monotone in $p$, with no spikes at either interpolation threshold. As in the ridgeless case, the ratio approaches $1$ as $p\gg n$, with both methods converging to the null risk.

Variations in cluster structure have essentially no effect on the comparison. Across all four panels of Figure~\ref{fig:F1}, the curves for the different configurations overlap so tightly that the differences sit below the $95\%$ confidence intervals at almost every $p$. To probe whether this depends on partitioning along the true cluster structure, we ran an ablation in which data were generated under the two-cluster framework of Section~\ref{sec:sim-setup} but randomly split into two halves to train the ensemble base learners. For both ridge and ridgeless regression, the results were identical to Figure~\ref{fig:F1}, even though the random partitions bore no relation to the true clusters. We further considered a fully non-clustered regime aligned with \cite{patil2023bagging}, in which both training and test data were drawn from a $p$-dimensional standard multivariate normal. Splitting the training data into two random equal halves for the ensemble, with the pooled learner trained on the full set, gave the same raw MSEs, and therefore the same MSE ratio, as the clustered experiments. The salient feature of the partition is not its alignment with any cluster structure but simply that it splits the data into two equal halves so that $2\times p/n_\ell=p/n$. The dominant axis of variation is the aspect ratio; the geometry of the cluster means, and indeed the very presence of covariate heterogeneity, is by comparison irrelevant.

\paragraph{Limitations and future work.}
Our analysis focuses on well-specified linear models with shared coefficients and a shared noise variance across clusters. This is the appropriate setting for isolating pure covariate heterogeneity, but it does not cover model misspecification, nonlinear predictors, or clusters with different noise levels or signal-to-noise ratios. Several extensions are natural. The random matrix arguments developed here extend to joint shifts in mean and variance, beyond the mean shifts we consider. The case of estimated rather than known cluster memberships is open, as is the move beyond linear regression, which can follow the template of \cite{hastie2019surprises} where results were first established under linear models before being extended to general nonlinear settings. Theoretical analysis of the ridgeless estimator near the interpolation peaks under covariate shift is an additional future direction. Finally, when cluster effects are fixed rather than random, the optimal choice between pooled and ensemble strategies depends on unknown cluster parameters, making adaptive selection an interesting open question.

\newpage
\bibliographystyle{plainnat}
\bibliography{bibliography}

%%%%%%%%%%%%%%%%%%%%%%%%%%%%%%%%%%%%%%%%%%%%%%%%%%%%%%%%%%%%

\appendix

\section{Appendix}

\subsection{Further Related Work}
\label{app:related-work}

\paragraph{An unsupervised analog: stacking versus per-source SVD.}
\citet{baharav2025stackedsvd} study a closely related question in an unsupervised setting: given multiple datasets that share a common right singular subspace but differ in signal strength and noise level, they investigate whether it is better to stack all datasets and run a single SVD (\emph{stacked SVD}) or to run SVD on each dataset separately and aggregate the per-source estimates (\emph{SVD stacked}). Their main finding is that the (\emph{stacked SVD}) approach dominates when dataset-specific weights are chosen optimally (with optimal weights downweighting noisier or lower-signal sources), but that without optimal weighting, neither approach dominates universally and per-source aggregation can occasionally win. Their setting differs from ours in the source of heterogeneity: all of their datasets share the same isotropic noise structure and signal direction but differ in signal-to-noise ratio, whereas in our supervised setting all clusters share the same signal and the noise level but differ in their covariate distributions. Despite this difference, the two settings raise the same fundamental question: when multiple data sources differ in some structural way,  whether it is optimal to combine them before fitting or to fit separately and aggregate. 
\subsection{Background Results} Throughout we repeatedly use the following anisotropic local laws \citep{knowles2017anisotropic}.
\begin{proposition}[\citep{knowles2017anisotropic}]\label{prop:prop_anisotropic_local_law}
Consider $n$ i.i.d. vectors    $\bX_i = \Sigma^{1/2}\mathbf{Z}_i $ with $\mathbf{Z}_i=(Z_{ij})_{j=1}^p$ are iid from some distribution with mean $0$, variance $1$ and sub-gaussian tails. With $\hat{\Sigma}=\frac{1}{n}\sum\limits_{i=1}^n \bX_i\bX_i^\top$ and constants $\alpha_0,\alpha_1>0$ define the event
    $$
    \Omega_n(\alpha_0,\alpha_1) = \left\{ \left| v^\top (\hat \Sigma + \lambda \mathbb{I})^{-1}v - v^\top \left(\mathbb{I} + \tilde m_n(-\lambda) \Sigma\right)^{-1}v \right| \le \frac{1}{n^{\frac12 - \alpha_0}\lambda}\right\}
    $$
    for all $\lambda \ge n^{-\frac23 + \alpha_1}$.
    Here $\tilde m_n(z)$ is the Stieltjes transform of the companion of the Marchenko-Pasteur law, defined as the solution of the following fixed-point equation: 
    \begin{equation}
    \label{eq:comp_MP}
    \frac{1}{\tilde m_n(z)} = - z + \gamma \frac1p\sum_{j = 1}^p \frac{s_j(\Sigma)}{1 + s_j(\Sigma)\tilde m_n(z)} = -z + \gamma \int \frac{x}{1 + x \tilde m_n(z)} \ dF_{\Sigma, n}(x) \,,
    \end{equation}
    where $\{s_j(\Sigma)\}_{j \in [p]}$ are eigenvalues of $\Sigma$ and $\gamma=\lim \frac{p}{n}$. 
    %By anisotropic local law (Theorem 3.17 of \cite{knowles2017anisotropic}, see also proof of Theorem 5 of \cite{hastie2022surprises}), 
    Then given any large $D > 0$ and small $\alpha_0, \alpha_1 > 0$, there exists $N_0 \geq 1$, such that $\mathbb{P}(\Omega_n^c(\alpha_0,\alpha_1)) \le n^{-D}$ for $n \ge N_0$.
\end{proposition}

\subsection{Proof of Main Results} Throughout we will assume $\tau^2=\sigma_{\varepsilon}^2=1$. The general case follows verbatim through scaling.
\subsubsection{Proof of Theorem \ref{thm:Pooled}}

Recall the notation and definitions as follows.\\
$\hat{\Sigma}_{\ell}=\frac{1}{n_\ell}\mathbb{X}_{\ell}^\top\mathbb{X}_{\ell}$ and $\hat{\Sigma}_{\rm pool}= \frac{1}{n}\mathbb{X}_{\rm pool}^\top \mathbb{X}_{\rm pool}$ where $n=\sum\limits_{\ell=1}^Kn_{\ell}$. Then define ridge regression as:
\begin{align*}
    \hat{f}_{\ell}(\bx)&=\hat{\boldsymbol{\beta}}_{\lambda_{\ell}}^{\top}\bx,\quad \hat{\boldsymbol{\beta}}_{\lambda_{\ell}}=(\mathbb{X}_{\ell}^\top\mathbb{X}_{\ell}+\lambda_{\ell}\mathbb{I})^{-1}\mathbb{X}_{\ell}^{\top}\mathbf{Y}_{\ell},\quad \lambda_{\ell}>0;\\
    \hat{f}_{\rm pool}(\bx)&=\hat{\boldsymbol{\beta}}_{\lambda}^{\top}\bx,\quad \hat{\boldsymbol{\beta}}_{\lambda}=(\mathbb{X}_{\rm pool}^\top\mathbb{X}_{\rm pool}+\lambda\mathbb{I})^{-1}\mathbb{X}_{\rm pool}^{\top}\mathbf{Y}_{\rm pool},\quad \lambda>0.
\end{align*}
Define ridge-less regression as:
\begin{align*}
    \hat{f}_{\ell,+}(\bx)&=\hat{\boldsymbol{\beta}}_{+,\ell}^{\top}\bx,\quad \hat{\boldsymbol{\beta}}_{+,\ell}=(\mathbb{X}_{\ell}^\top\mathbb{X}_{\ell})^{+}\mathbb{X}_{\ell}^{\top}\mathbf{Y}_{\ell};\\
    \hat{f}_{\rm pool, +}(\bx)&=\hat{\boldsymbol{\beta}}_{+}^{\top}\bx,\quad \hat{\boldsymbol{\beta}}_{+}=(\mathbb{X}_{\rm pool}^\top\mathbb{X}_{\rm pool})^{+}\mathbb{X}_{\rm pool}^{\top}\mathbf{Y}_{\rm pool}.
\end{align*}
We first analyze $ \hat{f}_{\rm pool}(\bx)$. We divide the analysis in two cases and sub-cases. The first case pertains to the random effects analysis -- further divided into fixed $\rho$ analysis followed by the ridgeless. We subsequently follow the same template for the fixed effects analysis i.e. fixed $\rho$ analysis followed by the ridgeless. Each of the analysis follows from an initial template decomposition of the risk as follows. Specifically,  for any estimator $\hat{\bbeta}$ of $\bbeta$ the following holds under Assumption \ref{assumptions}
\begin{align*}
    \E(y^\star-\hat{\bbeta}^\top\bx^\star)^2=\sigma_{\star}^2+\E\|\hat{\bbeta}-\bbeta\|^2.
\end{align*}
Further each $\hat{\bbeta}$ in our analyses is a linear function of the outcome and can be written as $\mathbf{L}\mathbf{Y}_{\rm pool}$ for some matrix $\mathbf{L}$ that is a measurable function of $\mathbb{X}_{\rm pool}$. Therefore
\begin{align*}
\E\|\hat{\bbeta}-\bbeta\|^2&=\E\|(L\mathbb{X}_{\rm pool}-\I)\bbeta\|^2+\tr(\mathbf{L}\mathbf{L}^\top).
    \end{align*}
The difference in analysis for fixed and random effects will pertain to evaluating the expectation in $\E\|(L\mathbb{X}_{\rm pool}-\I)\bbeta\|^2$ -- for random effects this marginalizes over the distribution of $\bbeta$ compared to analyzing for a given $\bbeta$ in the fixed effects case. 

\paragraph{Analysis of Pooled Ridge-Regularized Estimator -- Random Effects:} Here $\mathbf{L}=(\Xm^\top\Xm+\lambda\I)^{-1}\Xm^\top$. Moreover since $\E(\bbeta)=0$ and $\mathrm{Var}(\bbeta)=\frac{\I}{\sqrt{p}}$ we have by direct calculations that
\begin{align*}
    \E\|(L\mathbb{X}_{\rm pool}-\I)\bbeta\|^2&=\frac{1}{p}\tr((L\mathbb{X}_{\rm pool}-\I)^\top (L\mathbb{X}_{\rm pool}-\I))\\
    &=\rho^2\frac 1p\tr (G_{\rm pool}^2(-\rho)),
\end{align*}
where $\rho=\lambda/n$ and $G_{\rm pool}(z)=(\hat{\Sigma}_{\rm pool}-z\I)^{-1}$ for $z\notin \mathrm{Spectrum}(\hat{\Sigma}_{\rm pool})$. Similarly,
\begin{align*}
   \tr(\mathbf{L}\mathbf{L}^\top)=\frac{\gamma_0}{p}\E\left(\tr(G_{\rm pool}(-\rho)-\rho G_{\rm pool}^2(-\rho))\right).
\end{align*}

Next we analyze we analyze $\frac 1p\tr (G_{\rm pool}^2(-\rho))$. To this end note that, we can write $\Xm=\mathbb{Z}_{\rm pool}\Sigma^{1/2}+M$ where $\mathrm{rank}(M)\leq K$. Specifically $M=DU$ where $D\in \mathbb{R}^{n\times K}$ with $\ell^{\rm th}$ column being a vector of $1$'s of length $n_{\ell}$ starting at $\sum_{j=1}^{\ell-1} n_j+1$ and rest $0$'s, and $U=[\bmu_1,\ldots,\bmu_K]^\top\in \mathbb{R}^{K \times p}$. In this notation, we can write $$\hat{\Sigma}_{\rm pool}=\hat{\Sigma}_{\rm pool,0}+\frac{1}{n}\mathbb{Z}_{\rm pool,0}^T M+\frac{1}{n}M^\top\mathbb{Z}_{\rm pool,0}+\frac{1}{n}M^\top M$$ where $\mathbb{Z}_0=\mathbb{Z}_{\rm pool}\Sigma^{1/2}$ and $\hat{\Sigma}_{\rm pool,0}=\frac 1n \mathbb{Z}_0^\top \mathbb{Z}_0$. Therefore $\mathrm{rank}(\hat{\Sigma}_{\rm pool}-\hat{\Sigma}_{\rm pool,0})\leq 3K$. Now let $\hat{F}_{\rm pool}$ and $\hat{F}_{\rm pool,0}$ denote the empirical spectral distributions of  $\hat{\Sigma}_{\rm pool}$, and $\hat{\Sigma}_{\rm pool,0}$ respectively. Then \citep[Theorem A.43]{bai2010spectral} one has that
\begin{align*}
    \sup\limits_{x\in \mathbb{R}_+}|\hat{F}_{\rm pool}(x)-\hat{F}_{\rm pool,0}(x)|\leq \frac{3K}{p}\to 0.
\end{align*}
Next note that the functions $f_{t,\rho}(x)=\frac{1}{(x+\rho)^t}$ for $t\in \mathbb{N}$ is bounded and continuous on $\mathbb{R}_+$ for any fixed $\rho>0$ with $\|f_{t,\rho}\|_{\infty}\leq \frac{1}{\rho}$. Therefore, 
\begin{align*}
    \frac{1}{p}\tr(G^t_{\rm pool,0}(-\rho))\stackrel{\P}{\to} \int f_{t,\rho}(x)dF_{\gamma_0,H}
\end{align*}
where $H$ is the limiting spectral distribution of $\Sigma$ and $F_{\gamma_0,H}$ is the generalized Marchenko-Pastur distribution \citep{bai2010spectral}. However, 
\begin{align*}
    |\frac{1}{p}\tr(G^t_{\rm pool,0}(-\rho))-\frac{1}{p}\tr(G^t_{\rm pool}(-\rho))|=|\int f_{t,\rho}(x)d\hat{F}_{\rm pool}(x)-\int f_{t,\rho}(x)d\hat{F}_{\rm pool,0}(x) |\leq \frac{1}{\rho}\times \frac{3K}{p}\to 0.
\end{align*}
Hence,
\begin{align*}
    \frac{1}{p}\tr(G^t_{\rm pool}(-\rho))\stackrel{\P}{\to} \int f_{t,\rho}(x)dF_{\gamma_0,H}(x).
\end{align*}
Therefore
\begin{align*}
    \E\|\hat{\bbeta}_{\lambda}-\bbeta\|^2\stackrel{\P}{\to} \rho^2 \int f_{2,\rho}(x)dF_{\gamma_0,H}(x)+\gamma_0\left[\int f_{1,\rho}(x)dF_{\gamma_0,H}(x)-\rho \int f_{2,\rho}(x)dF_{\gamma_0,H}(x)\right].
\end{align*}
Now in terms of the notation of the theorem, it follows by direct calculations that
\begin{align*}
    \int f_{1,\rho}(x)dF_{\gamma_0,H}(x)&=a_{\gamma}(\rho);\\
    \int f_{2,\rho}(x)dF_{\gamma_0,H}(x)&=-a^{'}_{\gamma}(\rho).
\end{align*}
This completes the proof for the merged ridge regression when $\rho>0$.

\paragraph{Analysis of Pooled Ridgeless Estimator -- Random Effects:} In this case, we start by noting that
\begin{align*}
    \mathbf{L}\mathbf{L}&=\widehat{\Sigma}_{\rm pool}^+;\\
    \E\|(L\mathbb{X}_{\rm pool}-\I)\bbeta\|^2&=\frac{1}{p}\tr(\I-P_{\mathcal{C}(\Xm)}),
\end{align*}
where for any matrix $F$, $P_{\mathcal{C}(F)}$ stands for the projector onto the column space  $\mathcal{C}(F)$ of $F$.
Subsequently we divide the analysis into two parts, depending on $\gamma_0$ is smaller or larger than $1$ i.e. under versus overparametrized regimes. 

\subparagraph{Underparametrized Regime - $\gamma_0<1$:} We borrow the notation of the subsection and note that in this regime $\hat{\Sigma}_{\rm pool}^+=\hat{\Sigma}_{\rm pool}^{-1}$ (henceforth we workn on the event that $\hat{\Sigma}$ is invertible, that occurs with probability converging to 1) and $\hat{\Sigma}_{\rm pool,0}^+=\hat{\Sigma}_{\rm pool,0}^{-1}$. Subsequently, using the inverse identity that $A^{-1}-B^{-1}=-A^{-1}(B-A)B^{-1}$ we get that
\begin{align*}
    \hat{\Sigma}_{\rm pool}^{-1}-\hat{\Sigma}_{\rm pool,0}^{-1}&=- \hat{\Sigma}_{\rm pool}^{-1}\left(\frac{1}{n}\mathbb{Z}_{\rm pool,0}^T M+\frac{1}{n}M^\top\mathbb{Z}_{\rm pool,0}+\frac{1}{n}M^\top M\right)\hat{\Sigma}_{\rm pool,0}^{-1}.
\end{align*}
Therefore
\begin{align*}
    \mathrm{rank}(\hat{\Sigma}_{\rm pool}^{-1}-\hat{\Sigma}_{\rm pool,0}^{-1})\leq 3K.
\end{align*}
Moreover, there exists $c_{\gamma_0}<\infty$ such that with probability converging to $1$ \citep{bai2010spectral} one has that
\begin{align*}
    \|\hat{\Sigma}_{\rm pool}^{-1}-\hat{\Sigma}_{\rm pool,0}^{-1}\|_{\rm op}\leq c_{\gamma_0}.
\end{align*}
Therefore with probability converging to $1$ 
\begin{align*}
    |\frac{1}{p}\tr(\widehat{\Sigma}_{\rm pool}^+)-\frac{1}{p}\tr(\widehat{\Sigma}_{\rm pool,0}^+)|\leq \frac{1}{p}\mathrm{rank}(\hat{\Sigma}_{\rm pool}^{-1}-\hat{\Sigma}_{\rm pool,0}^{-1})\|\hat{\Sigma}_{\rm pool}^{-1}-\hat{\Sigma}_{\rm pool,0}^{-1}\|_{\rm op}\leq \frac{3Kc_{\gamma_0}}{p}\to 0.
\end{align*}
Moreover, in this regime $\frac{1}{p}\tr(\I-P_{\mathcal{C}(\Xm)})=0.$ However, in this regime \citep{hastie2019surprises},
\begin{align*}
    \frac{1}{n}\tr(\widehat{\Sigma}_{\rm pool,0}^+)\stackrel{\P}{\to} \gamma_0\int_{x>0}\frac{1}{x}dF_{\gamma_0,H}(x).
\end{align*}
This implies that
\begin{align*}
    \E\|\hat{\bbeta}-\bbeta\|^2 \stackrel{\P}{\to} \gamma_0\int_{x>0}\frac{1}{x}dF_{\gamma_0,H}(x)\textcolor{black}{=\lim_{\rho_0\downarrow 0}\left(\rho_0^2[-a_{\gamma_0}'(\rho_0)]
    +
    V_{\gamma_0}(\rho_0)\right)},
\end{align*}
as promised.

\subparagraph{Overparametrized Regime - $\gamma_0>1$:} Let $\Gamma=\frac{1}{n}\Xm\Xm^\top$ and $\Gamma_0=\frac{1}{n}\mathbb{Z}_0\mathbb{Z}_0^\top$ where $\mathbb{Z}_0=\mathbb{Z}_{\rm pool}\Sigma^{1/2}$. Now note that
\begin{align*}
    \tr(\hat{\Sigma}^+_{\rm pool})=\tr(\Gamma^{-1}),\quad  \tr(\hat{\Sigma}^+_{\rm pool,0})=\tr(\Gamma_0^{-1}).
\end{align*}
Arguing similar to before,
\begin{align*}
    \mathrm{rank}(\Gamma-\Gamma_0)\leq 3K,
\end{align*}
and there exists $c_{\gamma_0}<\infty$ such that with probability converging to $1$
\begin{align*}
    \|\Gamma\|_{\rm op}\vee \|\Gamma_0\|_{\rm op}\leq c_{\gamma_0}.
\end{align*}
Therefore
\begin{align*}
    |\frac{1}{p}\tr (\hat{\Sigma}_{\rm pool}^+)-\frac{1}{p}\tr (\hat{\Sigma}_{\rm pool,0}^+)|\leq 6Kc_{\gamma_0}/p\to 0.
\end{align*}
Therefore
\begin{align*}
    \frac{1}{p}\tr (\hat{\Sigma}_{\rm pool}^+)\stackrel{\P}{\to}\gamma_0\int_{x>0}\frac{1}{x}dF_{\gamma_0,H}(x),
\end{align*}
as before. Moreover it is easy to check by simple rank calculations that
\begin{align*}
    |\frac{1}{p}\tr(\I-P_{\mathcal{C}(\Xm)})-\frac{1}{p}\tr(\I-P_{\mathcal{C}(\mathbb{Z}_0)})|\leq K/p\to 0.
\end{align*}
But $\frac{1}{p}\tr(\I-P_{\mathcal{C}(\mathbb{Z}_0)})=1-n/p\to 1-1/\gamma_0$. Therefore as before
\begin{align*}
    \E\|\hat{\bbeta}-\bbeta\|^2 \stackrel{\P}{\to} \gamma_0\int_{x>0}\frac{1}{x}dF_{\gamma_0,H}(x)\textcolor{black}{=\lim_{\rho_0\downarrow 0}\left(\rho_0^2[-a_{\gamma_0}'(\rho_0)]
    +
    V_{\gamma_0}(\rho_0)\right)},
\end{align*}
as promised.

\paragraph{Analysis of Pooled Ridge-Regularized Estimator -- Fixed Effects:} First note that by direct calculations
\begin{align*}
    \E\|\hat{\bbeta}_{\lambda}-\bbeta\|^2=\sigma_\star^2+\rho_0^2\bbeta^\top G_{\rm pool}(-\rho_0)\bbeta+\frac{1}{n}\tr(\hat{\Sigma}_{\rm pool}G_{\rm pool}(-\rho_0))=\sigma_\star^2+B_n+V_n \quad (\text{say}).
\end{align*}
The analysis of $V_n\stackrel{\P}{\to}V_{\gamma_0}(\rho_0)$ is as before, and we only analyze the bias $B_n$ below.
Write as before
\[
\Xm=\mathbb{Z}_0+M,
\]
For each cluster define
\[
\bar z_\ell
=
\frac1{n_\ell}\sum_{i=1}^{n_\ell}\bZ_{i,\ell},
\qquad
\widetilde Z_{i,\ell}
=
\bZ_{i,\ell}-\bar z_\ell.
\]
Let
\[
S_z
=
\frac1n
\sum_{\ell=1}^K
\sum_{i=1}^{n_\ell}
\Sigma^{1/2}\bZ_{i,\ell}\bZ_{i,\ell}^\top\Sigma^{1/2},
\]
and
\[
S_n^{(-)}
=
\frac1n
\sum_{\ell=1}^K
\sum_{i=1}^{n_\ell}
\Sigma^{1/2}
\widetilde \bZ_{i,\ell}
\widetilde \bZ_{i,\ell}^{\top}
\Sigma^{1/2}.
\]
Recall
\[
\pi_{\ell,n}=\frac{n_\ell}{n},
\qquad
U_n=
[
\sqrt{\pi_{1,n}}\mu_1,\ldots,\sqrt{\pi_{K,n}}\mu_K
],
\]
and define
\[
R_n=
[
\sqrt{\pi_{1,n}}\Sigma^{1/2}\bar z_1,\ldots,
\sqrt{\pi_{K,n}}\Sigma^{1/2}\bar z_K
].
\]
Then
\[
S_z
=
S_n^{(-)}+R_nR_n^\top.
\]
Also, since inside cluster \(\ell\),
\[
\bX_{i,\ell}
=
\bmu_\ell+\Sigma^{1/2}\bar z_\ell
+
\Sigma^{1/2}\widetilde \bZ_{i,\ell},
\]
and
\[
\sum_{i=1}^{n_\ell}\widetilde Z_{i,\ell}=0,
\]
we get the exact decomposition
\[
\widehat\Sigma_{\rm pool}
=
S_n^{(-)}+(U_n+R_n)(U_n+R_n)^\top.
\]
Now for \(\rho>0\), define the two resolvents
\[
H_n(\rho)
=
(S_z+\rho \I)^{-1},
\qquad
G_n^{(-)}(\rho)
=
(S_n^{(-)}+\rho \I)^{-1}.
\]
%The proof first controls bilinear forms involving \(R_n\) and \(H_n(\rho)\) by row-level leave-one-out and leave-two-out arguments, and then converts those limits to corresponding bilinear forms involving \(G_n^{(-)}(\rho)\) using Woodbury.
\begin{lemma}\label{lemma:hn}
    Fix \(\rho\) in a compact subset of \((0,\infty)\). Then, locally uniformly in \(\rho\),
\[
R_n^\top H_n(\rho)R_n
\overset{\P}{\longrightarrow}
\frac{\gamma_0s_{\gamma_0}(\rho)}
{1+\gamma_0s_{\gamma_0}(\rho)}
\I_K,
\]
where $s_\gamma(\rho)
    =
    \frac1p\tr\{\Sigma Q_\gamma(\rho)\},$ and $\I_K$ is the identity matrix of order $K$.
Moreover, for every deterministic \(v\in\mathbb R^p\) with uniformly bounded norm,
\[
v^\top H_n(\rho)R_n
\overset{\P}{\longrightarrow}0,
\qquad
R_n^\top H_n(\rho)v
\overset{\P}{\longrightarrow}0.
\]

\end{lemma}
\begin{proof}
    Fix a cluster \(\ell\), write \(m=n_\ell\), and suppress the cluster index and dependence on resolvents on $\rho$ whenever clear from context. Let
\[
\bar z=\frac1m\sum_{i=1}^m z_i,
\qquad
r_{\ell,n}
=
\sqrt{\pi_{\ell,n}}\Sigma^{1/2}\bar z.
\]
Then
\[
r_{\ell,n}^\top H_n r_{\ell,n}
=
\pi_{\ell,n}
\bar z^\top
\Sigma^{1/2}H_n\Sigma^{1/2}
\bar z.
\]
Therefore
\[
\bar z^\top
\Sigma^{1/2}H_n\Sigma^{1/2}
\bar z
=
\frac1{m^2}
\sum_{i,j=1}^m
\bZ_i^\top
\Sigma^{1/2}H_n\Sigma^{1/2}
\bZ_j.
\]
We first treat the diagonal terms i.e. $i=j$. Let \(H_n^{(-i)}\) be the resolvent with the row \(\bZ_i\) removed from \(S_z\). That is,
\[
H_n^{(-i)}
=
\left(
S_z-\frac1n\Sigma^{1/2}\bZ_i\bZ_i^\top\Sigma^{1/2}
+\rho \I
\right)^{-1}.
\]
Now let
\[
B_i=\Sigma^{1/2}H_n^{(-i)}\Sigma^{1/2},
\qquad
q_i=\bZ_i^\top B_i \bZ_i.
\]
Then by Sherman-Morrison formula,
\[
H_n
=
H_n^{(-i)}
-
\frac{
\frac1nH_n^{(-i)}
\Sigma^{1/2}\bZ_i\bZ_i^\top\Sigma^{1/2}
H_n^{(-i)}
}{
1+\frac1n \bZ_i^\top B_i \bZ_i
}.
\]
Therefore
\[
\bZ_i^\top\Sigma^{1/2}H_n\Sigma^{1/2}\bZ_i
=
\frac{q_i}{1+q_i/n}.
\]
now note that \(H_n^{(-i)}\) is independent of \(z_i\) and there exists a constant $C_{\rho}<\infty$ such that with probability converging to $1$ \citep{bai2010spectral}
\[
\|H_n^{(-i)}\|_{\rm op}\le \rho^{-1},
\qquad
\|B_i\|_{\rm op}\le C_\rho,
\]
Therefore by Hanson--Wright inequality we have by union bound that uniformly over \(i\) 
\[
q_i
=
\tr B_i
+
O_{\P}(\sqrt{p\log n}).
\]

Also, removing one row changes normalized traces by \(o_{\P}(1)\) uniformly over rows \citep[Theorem A.43]{bai2010spectral}. Therefore uniformly over $i$
\[
\frac1p\tr B_i
=
\frac1p\tr\{\Sigma^{1/2}H_n^{(-i)}\Sigma^{1/2}\}
\stackrel{\P}{\to}
s_{\gamma_0}(\rho).
\]
Consequently,
\[
\frac1{m^2}
\sum_{i=1}^m
\bZ_i^\top\Sigma^{1/2}H_n\Sigma^{1/2}\bZ_i
=
\frac1{m^2}
\sum_{i=1}^m
\frac{q_i}{1+q_i/n}
\]
satisfies
\[
\frac1{m^2}
\sum_{i=1}^m
\frac{q_i}{1+q_i/n}
-
\frac1m
\cdot
\frac{p\,s_{\gamma_0}(\rho)}
{1+\gamma_0s_{\gamma_0}(\rho)}\stackrel{\P}{\to} 0.
\]
Since
\[
\frac pm
=
\frac{p/n}{n_\ell/n}
\to
\frac{\gamma_0}{\pi_\ell},
\]
we obtain
\[
\frac1{m^2}
\sum_{i=1}^m
\bZ_i^\top\Sigma^{1/2}H_n\Sigma^{1/2}\bZ_i
\stackrel{\P}{\to}
\frac{\gamma_0}{\pi_\ell}
\frac{s_{\gamma_0}(\rho)}
{1+\gamma_0s_{\gamma_0}(\rho)}.
\]
Next, we consider the off-diagonal terms \(i\ne j\). Here we employ a leave-two-out technique as follows. We start by leaving out the index $i$ and write using the Sherman-Morrison formula
\begin{align*}
\bZ_i^\top\Sigma^{1/2}H_n\Sigma^{1/2}\bZ_i=\frac{\bZ_i^\top \Sigma^{1/2}H_n^{(-i)}\Sigma^{1/2}\bZ_j}{1+\frac{1}{n}\bZ_i^\top\Sigma^{1/2} H_n^{(-i)}\Sigma^{1/2}\bZ_i}.
\end{align*}
Now, by applying the Sherman-Morrison formula once more to leave out the index $j$ we have
\begin{align*}
  \bZ_i^\top\Sigma^{1/2} H_n^{(-i)}\bZ_j\Sigma^{1/2}&=\frac{\bZ_i^\top\Sigma^{1/2} H_n^{(-ij)}\bZ_j\Sigma^{1/2}}{1+\bZ_j^\top \Sigma^{1/2}H_n^{(-i,j)}\bZ_j\Sigma^{1/2}}, 
\end{align*}
where
\[
H_n^{(-ij)}
=
\left(
S_z-\frac1n\Sigma^{1/2}\bZ_i\bZ_i^\top\Sigma^{1/2}-\frac1n\Sigma^{1/2}\bZ_j\bZ_j^\top\Sigma^{1/2}
+\rho \I
\right)^{-1}.
\]
Therefore
\begin{align*}
    \bZ_i^\top\Sigma^{1/2}H_n\Sigma^{1/2}\bZ_i=\frac{\bZ_i^\top\Sigma^{1/2} H_n^{(-ij)}\bZ_j\Sigma^{1/2}}{(1+\frac{1}{n}\bZ_i^\top\Sigma^{1/2} H_n^{(-i)}\bZ_i\Sigma^{1/2})(1+\frac{1}{n}\bZ_j^\top\Sigma^{1/2} H_n^{(-ij)}\bZ_j\Sigma^{1/2})}
\end{align*}
Therefore, we essentially need to control
\begin{align*}
    \frac{1}{m^2}\sum\limits_{i\neq j=1}^m \frac{\bZ_i^\top\Sigma^{1/2} H_n^{(-ij)}\bZ_j\Sigma^{1/2}}{(1+\frac{1}{n}\bZ_i^\top\Sigma^{1/2} H_n^{(-i)}\bZ_i\Sigma^{1/2})(1+\frac{1}{n}\bZ_j^\top\Sigma^{1/2} H_n^{(-ij)}\bZ_j\Sigma^{1/2})}.
\end{align*}
%By independence of $\bZ_i$ and $H_n^{(-i)}$ and that of $\bZ_j$ and $H_n^{(-i,j)}$ we have by Hanson-Wright inequality and union bound that their exists a constant $C>0$ (depending on $\gamma_0,\rho,H$ where $H$ is the limiting empirical spectral distribution of $\Sigma$) such that  
%\begin{align*}
 %   \ &\frac{1}{m^2}\sum\limits_{i\neq j=1}^m \frac{\bZ_i^\top\Sigma^{1/2} H_n^{(-i,j)}\bZ_j\Sigma^{1/2}}{(1+\frac{1}{n}\bZ_i^\top\Sigma^{1/2} H_n^{(-i)}\bZ_i\Sigma^{1/2})(1+\frac{1}{n}\bZ_j^\top\Sigma^{1/2} H_n^{(-i,j)}\bZ_j\Sigma^{1/2})}\\
  %  &= \frac{1}{m^2}\sum\limits_{i\neq j=1}^m \frac{\bZ_i^\top\Sigma^{1/2} H_n^{(-i,j)}\bZ_j\Sigma^{1/2}}{(1+C)^2} +o_{\P}(1).
%\end{align*}
%Now the first term of the last display has mean $0$. In order to compute the variance, we note that since each of these summands has standard deviation  $O(\sqrt{p})$ uniformly (this follows by Hanson-Wright inequality and the convergence to a generalized Marchenko-Pastur law of the empirical spectral distribution of $H_n^{(-i,j)}$ \citep{bai2010spectral}), it is enough to 
\[
\frac{1}{m^2}
\sum_{i\ne j}
\frac{
\bZ_i^\top \Sigma^{1/2}H_n^{(-ij)}\Sigma^{1/2}Z_j
}{
\left(1+n^{-1}\bZ_i^\top \Sigma^{1/2}H_n^{(-i)}\Sigma^{1/2}\bZ_i\right)
\left(1+n^{-1}\bZ_j^\top \Sigma^{1/2}H_n^{(-ij)}\Sigma^{1/2}\bZ_j\right)
}
=o_{\P}(1).
\]
Letting 
\[
B_{ij}=\Sigma^{1/2}H_n^{(-ij)}\Sigma^{1/2},
\]
we note that summands are not mutually independent, because \(B_{ij}\) depends on the
observations other than \(i,j\).  We handle this by decoupling.  Let \[
U_m=\sum_{i\ne j}\bZ_i^\top B_{ij}\bZ_j,
\qquad
B_{ij}=\Sigma^{1/2}H_n^{(-ij)}\Sigma^{1/2}.
\]
Since \(H_n^{(-ij)}\) removes rows \(i\) and \(j\), the matrix \(B_{ij}\) is
measurable with respect to \(\sigma(\bZ_k:k\ne i,j)\), and is therefore
independent of \((\bZ_i,\bZ_j)\).  Moreover, because \(\rho\) is bounded away from
zero and \(\|\Sigma\|_{\mathrm{op}}\) is bounded,
\[
\|B_{ij}\|_{\mathrm{op}}\le C_\rho,\qquad
\operatorname{tr}(B_{ij}^2)\le C_\rho p
\]
uniformly in \(i,j\), with probability tending to one. The issue is that the summands \(\xi_i^\top R_{ij}\xi_j\) are not mutually independent, because
\(R_{ij}\) depends on all rows except \(i,j\).  We therefore prove the required bound by a
leave-four covariance calculation.  Expanding the second moment gives
\[
\E U_m^2
=
\sum_{i\ne j}\sum_{k\ne \ell}
\E\left[
\xi_i^\top R_{ij}\xi_j\,
\xi_k^\top R_{k\ell}\xi_\ell
\right].
\]
We classify the terms according to the overlap between the ordered pairs \((i,j)\) and
\((k,\ell)\).

First, consider the diagonal and reversed-pair cases.  If \((i,j)=(k,\ell)\), then, conditional on
\(R_{ij}\),
\[
\E\left[(\xi_i^\top R_{ij}\xi_j)^2\mid R_{ij}\right]
\le C\,\operatorname{tr}\{(\Sigma^{1/2}R_{ij}\Sigma^{1/2})^2\}
\le Cp.
\]
The same bound holds when \((i,j)=(\ell,k)\).  Since there are \(O(m^2)\) such terms, their total
contribution is \(O(m^2p)\).

Next suppose that the two ordered pairs share exactly one index.  For concreteness, take
\(i=k\), with \(j\ne \ell\); the other cases are identical.  Let
\[
R_{ij\ell}:=H_n^{(-i,-j,-\ell)}.
\]
By the Sherman--Morrison identity,
\[
R_{ij}
=
R_{ij\ell}
-
\frac{n^{-1}R_{ij\ell}\xi_\ell\xi_\ell^\top R_{ij\ell}}
{1+n^{-1}\xi_\ell^\top R_{ij\ell}\xi_\ell}.
\]
Hence
\[
\xi_i^\top R_{ij}\xi_j
=
\xi_i^\top R_{ij\ell}\xi_j
-
\frac{
n^{-1}(\xi_i^\top R_{ij\ell}\xi_\ell)
(\xi_\ell^\top R_{ij\ell}\xi_j)
}{
1+n^{-1}\xi_\ell^\top R_{ij\ell}\xi_\ell
}.
\]
Similarly,
\[
R_{i\ell}
=
R_{ij\ell}
-
\frac{n^{-1}R_{ij\ell}\xi_j\xi_j^\top R_{ij\ell}}
{1+n^{-1}\xi_j^\top R_{ij\ell}\xi_j}.
\]
Therefore
\[
\xi_i^\top R_{i\ell}\xi_\ell
=
\xi_i^\top R_{ij\ell}\xi_\ell
-
\frac{
n^{-1}(\xi_i^\top R_{ij\ell}\xi_j)
(\xi_j^\top R_{ij\ell}\xi_\ell)
}{
1+n^{-1}\xi_j^\top R_{ij\ell}\xi_j
}.
\]
Conditional on \(R_{ij\ell}\), the vectors \(\xi_i,\xi_j,\xi_\ell\) are independent and centered.
Thus the product of the two leading terms has conditional expectation zero:
\[
\E\left[
(\xi_i^\top R_{ij\ell}\xi_j)
(\xi_i^\top R_{ij\ell}\xi_\ell)
\mid R_{ij\ell}
\right]=0.
\]
The remaining terms each contain at least one factor \(n^{-1}\).  Using
\(\|R_{ij\ell}\|_{\op}\le \rho^{-1}\), the boundedness of \(\|\Sigma\|_{\op}\), and the standard
quadratic-form bounds
\[
\E\left[(\xi_a^\top R_{ij\ell}\xi_b)^2\mid R_{ij\ell}\right]\le Cp
\quad(a\ne b),
\qquad
\E\left[\xi_a^\top R_{ij\ell}\xi_a\mid R_{ij\ell}\right]\le Cp,
\]
we obtain
\[
\left|
\E\left[
\xi_i^\top R_{ij}\xi_j\,
\xi_i^\top R_{i\ell}\xi_\ell
\right]
\right|
\le C\,\frac{p}{n}.
\]
Since there are \(O(m^3)\) ordered quadruples with exactly one shared index, the total contribution
of all one-overlap terms is
\[
O\left(m^3\frac{p}{n}\right).
\]

Finally suppose \(i,j,k,\ell\) are all distinct.  Let
\[
R_{ijk\ell}:=H_n^{(-i,-j,-k,-\ell)}.
\]
Apply Sherman--Morrison twice to express both \(R_{ij}\) and \(R_{k\ell}\) as rank-one updates of
the common leave-four resolvent \(R_{ijk\ell}\).  The leading product is
\[
(\xi_i^\top R_{ijk\ell}\xi_j)(\xi_k^\top R_{ijk\ell}\xi_\ell),
\]
whose conditional expectation given \(R_{ijk\ell}\) is zero, because
\(\xi_i,\xi_j,\xi_k,\xi_\ell\) are independent centered vectors.  Every nonzero contribution must
therefore come from resolvent-update terms.  In the disjoint case, to obtain a nonzero conditional
expectation, both factors must be updated: one update is needed to introduce \(\xi_k\) or \(\xi_\ell\)
into the first factor, and another update is needed to introduce \(\xi_i\) or \(\xi_j\) into the second
factor.  Thus every nonzero term carries at least two factors of \(n^{-1}\).  Using again the
boundedness of the leave-four resolvent and the quadratic-form moment bounds, each such term is
bounded by
\[
C\,\frac{p}{n^2}.
\]
Hence, for four distinct indices,
\[
\left|
\E\left[
\xi_i^\top R_{ij}\xi_j\,
\xi_k^\top R_{k\ell}\xi_\ell
\right]
\right|
\le C\,\frac{p}{n^2}.
\]
There are \(O(m^4)\) such quadruples, so their total contribution is
\[
O\left(m^4\frac{p}{n^2}\right).
\]

Combining the three cases,
\[
\E U_m^2
\le
C\left[
m^2p
+
m^3\frac{p}{n}
+
m^4\frac{p}{n^2}
\right].
\]
Since \(m/n\to \pi_\ell\in(0,1)\), this simplifies to
\[
E U_m^2\le C m^2p.
\]
Therefore
\[
U_m=O_{\P}(m\sqrt p),
\]
and hence
\[
\frac{1}{m^2}U_m
=
O_P\left(\frac{\sqrt p}{m}\right)
=o_P(1),
\]
because \(p/m\to \gamma_0/\pi_\ell<\infty\).  This proves that the off-diagonal contribution is
negligible.

Finally, for deterministic bounded \(v\),
\[
v^\top H_n(\rho)r_{\ell,n}
=
\sqrt{\pi_{\ell,n}}
v^\top H_n(\rho)\Sigma^{1/2}\bar z_\ell.
\]
Expanding the empirical mean and applying leave-one-out gives
\[
v^\top H_n\Sigma^{1/2}\bar z_\ell
=
\frac1m\sum_{i=1}^m
v^\top H_n^{(-i)}\Sigma^{1/2}\bZ_{i,\ell}
+
o_{\P}(1).
\]
Conditional on \(H_n^{(-i)}\), the summands have mean zero and uniformly bounded variance. Hence the average is \(O_{\P}(m^{-1/2})=o_{\P}(1)\). Therefore
\[
v^\top H_n(\rho)r_{\ell,n}\stackrel{\P}{\to} 0.
\]
This proves the lemma.
\end{proof}

The next lemma connects $H_n$ with $G_n^{(-)}$.
\begin{lemma}\label{lemma:hn_gn}
    Locally uniformly for \(\rho\) in compact subsets of \((0,\infty)\),
\[
R_n^\top G_n^{(-)}(\rho)R_n
\overset{\P}{\longrightarrow}
\gamma_0s_{\gamma_0}(\rho)I_K,
\]
and for deterministic bounded \(v\),
\[
v^\top G_n^{(-)}(\rho)R_n\overset{\P}{\longrightarrow}0,
\qquad
R_n^\top G_n^{(-)}(\rho)v\overset{\P}{\longrightarrow}0.
\]
\end{lemma}
\begin{proof}
    Recall
\[
S_z=S_n^{(-)}+R_nR_n^\top.
\]
Therefore
\[
H_n(\rho)
=
(S_z+\rho \I)^{-1}
=
(S_n^{(-)}+\rho \I+R_nR_n^\top)^{-1}.
\]
Therefore, Woodbury's identity gives
\[
H_n
=
G_n^{(-)}
-
G_n^{(-)}R_n
\left(\I_K+R_n^\top G_n^{(-)}R_n\right)^{-1}
R_n^\top G_n^{(-)}.
\]
Let
\[
T_n=R_n^\top G_n^{(-)}R_n,
\qquad
L_n=R_n^\top H_nR_n.
\]
Multiplying the Woodbury identity by \(R_n^\top\) and \(R_n\) gives
\[
L_n
=
T_n-T_n(I_K+T_n)^{-1}T_n
=
T_n(I_K+T_n)^{-1}.
\]
By Lemma \ref{lemma:hn},
\[
L_n
\stackrel{\P}{\to}
\frac{d}{1+d}\I_K,
\qquad
d=\gamma_0s_{\gamma_0}(\rho).
\]
Since
\[
T_n=L_n(\I_K-L_n)^{-1},
\]
we get
\[
T_n
\stackrel{\P}{\to}
\frac{d}{1+d}
\left(1-\frac{d}{1+d}\right)^{-1}
\I_K
=
d\I_K.
\]
Hence
\[
R_n^\top G_n^{(-)}R_n
\stackrel{\P}{\to}
\gamma_0s_{\gamma_0}(\rho)\I_K.
\]
For the mixed bilinear form, multiply the same Woodbury identity by \(v^\top\) and \(R_n\):
\[
v^\top H_nR_n
=
v^\top G_n^{(-)}R_n
-
v^\top G_n^{(-)}R_n(\I_K+T_n)^{-1}T_n.
\]
Thus
\[
v^\top H_nR_n
=
v^\top G_n^{(-)}R_n
\left[
\I_K-(\I_K+T_n)^{-1}T_n
\right]
=
v^\top G_n^{(-)}R_n(\I_K+T_n)^{-1}.
\]
Therefore
\[
v^\top G_n^{(-)}R_n
=
v^\top H_nR_n(\I_K+T_n).
\]
Lemma \ref{lemma:hn} gives \(v^\top H_nR_n\stackrel{\P}{\to}0\), and \(T_n=O_{\P}(1)\). Hence
\[
v^\top G_n^{(-)}R_n\stackrel{\P}{\to}0.
\]
The other statement follows identically. This proves the lemma.
\(\square\)
\end{proof}
We now complete the proof of the pooled ridge learner using the lemmas. By the anisotropic local law \citep{knowles2017anisotropic} and the fact that removing \(K\) empirical mean directions is finite rank, for deterministic bounded \(v_1,v_2\),
\[
v_1^\top G_n^{(-)}(\rho)v_2
-
v_1^\top Q_{\gamma_0}(\rho)v_2
\overset{\P}{\longrightarrow}0.
\]
Consequently,
\[
U_n^\top G_n^{(-)}(\rho)U_n
\stackrel{\P}{\to}
U_0^\top Q_{\gamma_0}(\rho)U_0,
\]
and by Lemma \ref{lemma:hn_gn},
\[
U_n^\top G_n^{(-)}(\rho)R_n\stackrel{\P}{\to}0,
\qquad
R_n^\top G_n^{(-)}(\rho)U_n\stackrel{\P}{\to}0.
\]
Next note that, since
\[
\widehat\Sigma_{\rm pool}
=
S_n^{(-)}+(U_n+R_n)(U_n+R_n)^\top,
\]
Woodbury's identity gives
\[
G_{\rm pool}(-\rho)
=
G_n^{(-)}(\rho)
-
G_n^{(-)}(\rho)(U_n+R_n)
D_n(\rho)^{-1}
(U_n+R_n)^\top G_n^{(-)}(\rho),
\]
where
\[
D_n(\rho)
=
\I_K+(U_n+R_n)^\top G_n^{(-)}(\rho)(U_n+R_n).
\]
Therefore, \[
(U_n+R_n)^\top G_n^{(-)}(\rho)(U_n+R_n)
\stackrel{\P}{\to}
U_0^\top Q_{\gamma_0}(\rho)U_0
+
\gamma_0s_{\gamma_0}(\rho)\I_K.
\]
Thus
\[
D_n(\rho)
\stackrel{\P}{\to}
\{1+\gamma_0s_{\gamma_0}(\rho)\}\I_K
+
U_0^\top Q_{\gamma_0}(\rho)U_0.
\]
Also, for deterministic norm-bounded \(v\),
\[
v^\top G_n^{(-)}(\rho)(U_n+R_n)
\stackrel{\P}{\to}
v^\top Q_{\gamma_0}(\rho)U_0,
\]
since
\[
v^\top G_n^{(-)}(\rho)R_n\stackrel{\P}{\to}0.
\]
Therefore
\[
v^\top G_{\rm pool}(-\rho)v
\stackrel{\P}{\to}
v^\top Q_{\gamma_0}(\rho)v
-
v^\top Q_{\gamma_0}(\rho)U_0
\left[
\{1+\gamma_0s_{\gamma_0}(\rho)\}\I_K
+
U_0^\top Q_{\gamma_0}(\rho)U_0
\right]^{-1}
U_0^\top Q_{\gamma_0}(\rho)v.
\]
However, by definition,
\[
K_{\gamma_0,U_0}(\rho)
=
Q_{\gamma_0}(\rho)
-
Q_{\gamma_0}(\rho)U_0
\left[
\{1+\gamma_0s_{\gamma_0}(\rho)\}\I_K
+
U_0^\top Q_{\gamma_0}(\rho)U_0
\right]^{-1}
U_0^\top Q_{\gamma_0}(\rho).
\]
Hence
\[
v^\top G_{\rm pool}(\rho)v
\stackrel{\P}{\to}
v^\top K_{\gamma_0,U_0}(\rho)v.
\]
We now analyze the bias term. We start by recalling
\[
\psi_0(\rho)=\beta^\top K_{\gamma_0,U_0}(\rho)\beta.
\]
Then as before
\[
\bbeta^\top G_{\rm pool}(-\rho)\bbeta
\stackrel{\P}{\to}
\psi_0(\rho)
\]
locally uniformly for \(\rho\) in compact subsets of \((0,\infty)\) since $\bbeta$ is norm bounded. Next, since
\[
\frac{d}{d\rho}G_{\rm pool}(-\rho)
=
-G_{\rm pool}(\rho)^2,
\]
we have
\[
\bbeta^\top G_{\rm pool}(-\rho)^2\bbeta
=
-\frac{d}{d\rho}
\left[
\bbeta^\top G_{\rm pool}(-\rho)\bbeta
\right].
\]
Since all the convergences are locally uniform in $\rho$ on compact sets, and all resolvents are uniformly bounded by \(\rho^{-1}\) on compact subsets of \((0,\infty)\), differentiating the deterministic equivalent gives
\[
\bbeta^\top G_{\rm pool}(-\rho_0)^2\bbeta
\stackrel{\P}{\to}
-\psi_0'(\rho_0).
\]
Therefore
\[
B_n
\stackrel{\P}{\to}
\rho_0^2[-\psi_0'(\rho_0)].
\]
This completes the proof for fixed effects pooled ridge estimator.

\paragraph{Analysis of Pooled Ridgeless Estimator -- Fixed Effects:}
Recall the notation
\[
\widehat\Sigma_{\rm pool}=\frac1nX^\top X,
\qquad
G_{\rm pool}(-\rho)=(\widehat\Sigma+\rho I)^{-1},
\qquad
\rho>0.
\]
Let the spectral decomposition of \(\widehat\Sigma_{\rm pool}\) be
\[
\widehat\Sigma_{\rm pool}
=
\sum_{j=1}^p s_j u_ju_j^\top,
\qquad
s_j\ge 0.
\]
Also note that
\[
P_{\mathcal{C}(\Xm)}=\widehat\Sigma_{\rm pool}\widehat\Sigma_{\rm pool}^+
\]
be the orthogonal projection onto the row space of \(X\). Then
\[
I-P_{\mathcal{C}(\Xm)}
=
\sum_{j:s_j=0}u_ju_j^\top.
\]

%Now for each fixed \(n,p\) and for every vector \(b\),
%\[
%b^\top(I-P_{\mathcal{C}(\Xm)})b
%=
%\lim_{\rho\downarrow0}
%\rho^2 b^\top G_{\rm pool}(-\rho)^2b.
%\]
%Similarly,
%\[
%\widehat\Sigma_{\rm pool} G_{\rm pool}(-\rho)^2
%=
%\sum_{j:s_j>0}
%\frac{s_j}{(s_j+\rho)^2}u_ju_j^\top,
%\]
%and therefore, provided the nonzero eigenvalues are bounded away from zero
%\[
%\frac1n\tr(\widehat\Sigma_{\rm pool}^+)
%=
%\lim_{\rho\downarrow0}
%\frac1n\tr\{\widehat\Sigma G_{\rm pool}(-\rho)^2\},
%\]
%with the convention that the right-hand side may diverge if the nonzero spectrum approaches zero.

%We now justify exchanging the \(\rho\downarrow0\) and \(n,p\to\infty\) limits away from the interpolation threshold. 

By the anisotropic local law \citep{knowles2017anisotropic} and smallest singular value bounds for sample covariance matrices with bounded finite-rank deterministic perturbations \citep{bai2010spectral}, there exists \(c_{\gamma_0}>0\) such that, with probability tending to one, every nonzero eigenvalue of \(\widehat\Sigma_{\rm pool}\) lies in
\[
[c_{\gamma_0},1/c_{\gamma_0}].
\]
%\textcolor{red}{The finite-rank non-centering perturbation may create only finitely many outlier eigenvalues, but away from \(\gamma=1\) it does not create a macroscopic mass of eigenvalues near zero. Thus the normalized trace and bounded deterministic bilinear forms are unaffected by such finite-rank perturbations.}
On this high-probability event, for any bounded norm deterministic \(b\)
\[
\left|
\rho^2 b^\top G_{\rm pool}(-\rho)^2b
-
b^\top(I-P_{\mathcal{C}(\Xm)})b
\right|
\le
\frac{\rho^2}{c_{\gamma_0}^2}\|b\|^2.
\]
Since for \(s_j\ge c_{\gamma_0}\),
\[
0\le
\frac{\rho^2}{(s_j+\rho)^2}
\le
\frac{\rho^2}{c_{\gamma_0}^2},
\]
we have that
\[
\lim_{n\to\infty}b^\top(I-P_{\mathcal{C}(\Xm)})b
=
\lim_{\rho\downarrow0}\lim_{n\to\infty}
\rho^2 b^\top G_{\rm pool}(-\rho)^2b,
\]
since the deterministic equivalent is locally uniform on compact subsets of $(0,\infty)$.
Similarly, for the variance term,
\[
\frac1n\tr(\widehat\Sigma^+)
-
\frac1n\tr\{\widehat\Sigma G_{\rm pool}(-\rho)^2\}
=
\frac1n\sum_{j:s_j>0}
\left[
\frac1{s_j}
-
\frac{s_j}{(s_j+\rho)^2}
\right].
\]
For \(s_j\ge c_{\gamma_0}\),
\[
0\le
\frac1{s_j}
-
\frac{s_j}{(s_j+\rho)^2}
=
\frac{2\rho s_j+\rho^2}{s_j(s_j+\rho)^2}
\le
c_{\gamma_0}\rho.
\]
Since the number of nonzero eigenvalues is at most \(n\wedge p\),
\[
\left|
\frac1n\tr(\widehat\Sigma_{\rm pool}^+)
-
\frac1n\tr\{\widehat\Sigma_{\rm pool} G_{\rm pool}(-\rho)^2\}
\right|
\le
c_{\gamma_0}\rho.
\]
Hence
\[
\frac1n\tr(\widehat\Sigma^+)
=
\lim_{\rho\downarrow0}\lim_{n\to\infty}
\frac1n\tr\{\widehat\Sigma G_{\rm pool}(-\rho)^2\}.
\]
Thus, the ridgeless bias and variance are obtained as the \(\rho\downarrow0\) limits of the corresponding ridge bias and variance deterministic equivalents i.e. 
\[
B_+
=
\lim_{\rho_0\downarrow0}
\rho_0^2[-\psi_0'(\rho_0)]
\]
 and 
\[
V_+
=
\lim_{\rho_0\downarrow0}
V_\gamma(\rho_0).
\]

\subsection{Proof of Theorem \ref{thm:gen_ensemble}}
Here, we only demonstrate the analysis for the random effects and ridge regularization. All the remaining proofs, i.e. random effects ridgeless and fixed effects analyses, follow by obvious modifications of the following and the proof of Theorem \ref{thm:Pooled} by taking $K=1$. 
For the random effects calculations for ridge regression, note that by direct calculations
\begin{align*}
    R_{\rm ens, \mathbf{w}}=\sigma_\star^2+\frac{1}{p}\tr(\sum w^2_{\ell}\rho^2_{\ell}G^2_{\ell}(-\rho_{\ell}))+\frac{1}{p}\tr(\sum_{\ell\neq \ell'} w_{\ell}w_{\ell'}\rho_{\ell}\rho_{\ell'}G_{\ell}(-\rho_{\ell})G_{\ell'}(-\rho_{\ell'})),
\end{align*}
where $G_{\ell}(-\rho_{\ell})=(\hat{\Sigma}_{\ell}+\rho_{\ell}\I)^{-1}$. The analysis of $+\frac{1}{p}\tr(\sum w^2_{\ell}\rho^2_{\ell}G^2_{\ell}(-\rho_{\ell}))$ is verbatim the same as in the proof of Theorem \ref{thm:Pooled} by taking $K=1$ and therefore we focus only on $\frac{\tau^2}{p}\tr(\sum_{\ell\neq \ell'} w_{\ell}w_{\ell'}\rho_{\ell}\rho_{\ell'}G_{\ell}(-\rho_{\ell})G_{\ell'}(-\rho_{\ell'}))$.
In the analysis, we drop the term $(-\rho_{\ell})$ from the notation of $G_{\ell}(-\rho_{\ell})$ whenever clear from context.

Fix \(\ell\neq j\). We claim that
\[
\frac{1}{p}\tr(G_\ell G_j)
=
\frac{1}{p}\tr(Q_\ell Q_j)
+
o_{\P}(1).
\]
To that end we write
\[
\mathbb{X}_\ell=X_\ell^0+\mathbf 1_{n_\ell}\bmu_\ell^\top,
\]
where $1_{n_\ell}$ is a $n_{\ell}\times 1$ vector of $1$'s, and the $i^{\rm th}$ row of \(X_\ell^0\) is \(\Sigma^{1/2}\bZ_{i,\ell}\). Let
\[
\widehat\Sigma_\ell^0
=
\frac{1}{n_\ell}(X_\ell^0)^\top X_\ell^0,
\qquad
G_\ell^0
=
(\widehat\Sigma_\ell^0+\rho_\ell \I)^{-1}.
\]
Then
\[
\widehat\Sigma_\ell-\widehat\Sigma_\ell^0
=
\frac{1}{n_\ell}(X_\ell^0)^\top \mathbf 1_{n_\ell}\bmu_\ell^\top
+
\frac{1}{n_\ell}\bmu_\ell\mathbf 1_{n_\ell}^\top X_\ell^0
+
\bmu_\ell\bmu_\ell^\top .
\]
Each term on the right-hand side has rank at most one, and hence
\[
\operatorname{rank}(\widehat\Sigma_\ell-\widehat\Sigma_\ell^0)\leq 3.
\]
By matrix inverse identity,
\[
G_\ell-G_\ell^0
=
-G_\ell(\widehat\Sigma_\ell-\widehat\Sigma_\ell^0)G_\ell^0.
\]
Therefore
\[
\operatorname{rank}(G_\ell-G_\ell^0)\leq 3.
\]
Moreover, since \(\rho_\ell>0\), 
\[
\|G_\ell\|_{\mathrm{op}}\leq \rho_\ell^{-1},
\qquad
\|G_\ell^0\|_{\mathrm{op}}\leq \rho_\ell^{-1}.
\]
Similarly,
\[
\operatorname{rank}(G_j-G_j^0)\leq 3,
\qquad
\|G_j\|_{\mathrm{op}}\leq \rho_j^{-1},
\qquad
\|G_j^0\|_{\mathrm{op}}\leq \rho_j^{-1}.
\]
Now
\[
G_\ell G_j-G_\ell^0G_j^0
=
(G_\ell-G_\ell^0)G_j
+
G_\ell^0(G_j-G_j^0).
\]
Thus
\[
\begin{aligned}
\left|
\frac{1}{p}\tr(G_\ell G_j)
-
\frac{1}{p}\tr(G_\ell^0G_j^0)
\right|
&\leq
\frac{1}{p}
\left|
\tr\{(G_\ell-G_\ell^0)G_j\}
\right|
\\
&\quad+
\frac{1}{p}
\left|
\tr\{G_\ell^0(G_j-G_j^0)\}
\right| .
\end{aligned}
\]
Using
\[
|\tr(AB)|
\leq
\operatorname{rank}(A)\|A\|_{\mathrm{op}}\|B\|_{\mathrm{op}},
\]
whenever \(A\) is finite-rank, together with the uniform boundedness of the resolvents, we get
\[
\left|
\frac{1}{p}\tr(G_\ell G_j)
-
\frac{1}{p}\tr(G_\ell^0G_j^0)
\right|
\leq C'(\rho_{\ell}^{-1},\rho_{j}^{-1})/p,
\]
for a finite constant $C'(\rho_{\ell}^{-1},\rho_{j}^{-1})$ depending on $\rho_{\ell}^{-1},\rho_{j}^{-1}$. Therefore
\[
\frac{1}{p}\tr(G_\ell G_j)
=
\frac{1}{p}\tr(G_\ell^0G_j^0)
+
o_{\P}(1),
\]
uniformly in $\ell,j$ (since $K$ is finite). It remains to replace the zero-mean resolvents by their deterministic equivalents. To that end we write
\[
\begin{aligned}
\frac{1}{p}\tr(G_\ell^0G_j^0)
-
\frac{1}{p}\tr(Q_\ell Q_j)
&=
\frac{1}{p}\tr\{(G_\ell^0-Q_\ell)G_j^0\}
\\
&\quad+
\frac{1}{p}\tr\{Q_\ell(G_j^0-Q_j)\}.
\end{aligned}
\]
Because the clusters are independent, conditional on \(G_j^0\), the matrix \(G_j^0\) is deterministic with respect to the randomness in \(G_\ell^0\). Moreover,
\[
\|G_j^0\|_{\mathrm{op}}\leq \rho_j^{-1}.
\]
Therefore, the anisotropic deterministic-equivalence \citep{knowles2017anisotropic} implies that for any deterministic matrix \(A_p\) with uniformly bounded operator norm,
\[
\frac{1}{p}\tr\{A_p(G_\ell^0-Q_\ell)\}
\stackrel{\P}{\to} 0.
\]
Applying this conditionally with \(A_p=G_j^0\) gives
\[
\frac{1}{p}\tr\{(G_\ell^0-Q_\ell)G_j^0\}
=
o_{\P}(1).
\]
Similarly, \(Q_\ell\) is deterministic and uniformly bounded in operator norm, and hence
\[
\frac{1}{p}\tr\{Q_\ell(G_j^0-Q_j)\}
=
o_{\P}(1).
\]
Consequently,
\[
\frac{1}{p}\tr(G_\ell^0G_j^0)
=
\frac{1}{p}\tr(Q_\ell Q_j)
+
o_{\P}(1).
\]
Combining this with the finite-rank perturbation bound yields
\[
\frac{1}{p}\tr(G_\ell G_j)
=
\frac{1}{p}\tr(Q_\ell Q_j)
+
o_{\P}(1).
\]
Therefore,
\[
\frac{1}{p}\tr(\sum_{\ell\neq \ell'} w_{\ell}w_{\ell'}\rho_{\ell}\rho_{\ell'}G_{\ell}(-\rho_{\ell})G_{\ell'}(-\rho_{\ell'}))
=
\sum_{\ell\neq j}
w_\ell w_j\rho_\ell\rho_j
\frac{1}{p}\tr(Q_\ell Q_j)
+
o_{\P}(1),
\]
as promised.

\subsection{Proof of Theorem \ref{thm:ensemble_optimal}}
The proof of part (i) follows from Theorem \ref{thm:gen_ensemble} by subsequently considering the optimal tuning result in \cite{dobriban2018high} since for the random effects model, the effect of the mean-shift does not reflect itself in the results owing to the lack of their effect asymptotically on the empirical spectral distributions of $\hat{\Sigma}_{\ell}$'s. The proof of part (ii) then follows from the standard Lagrange multiplier method and first-order condition analysis of the optimal weighting optimization problem over $\mathbf{w}\in \mathbb{R}_+$ such that $\mathbf{1}^\top\mathbf{w}=1$. Indeed, the unconstrained optimizer operates under the constraint $\mathbf{1}^\top\mathbf{w}=1$ using the Lagrange multiplier method to yield the desired result, and if an optimizer satisfies the non-negativity constraint, the result follows from standard analysis. 

\subsection{Proof of Corollary \ref{cor:comp_k2_iso_re}} We only verify the second claim and that $\mathcal{R}_{\rm excess}:=\frac{\kappa_1\kappa_2(1-\kappa_1\kappa_2)}
      {\kappa_0(\kappa_1+\kappa_2-2\kappa_1\kappa_2)}>1$.
The rest of the proof follows from standard algebra by evaluating the expressions in Theorem \ref{thm:ensemble_optimal} and Theorem \ref{thm:gen_ensemble_re} through integral w.r.t. classical Marchenko-Pastur distribution \citep[Section 3.1.1]{bai2010spectral}. We omit the algebraic details.

We now prove the claim that $\mathcal{R}_{\rm excess}>1$. First note that $\kappa(\gamma)$ is the positive root of
\[
    \kappa^2+
    \left(
        \alpha-1+\frac{1}{\gamma}
    \right)\kappa
    -
    \alpha
    =
    0.
\]
Since the polynomial is negative at $0$ and positive at $1$,
\[
    0<\kappa(\gamma)<1
\]
for every $\gamma>0$. Hence
\[
    0<\kappa_1,\kappa_2,\kappa_0<1.
\]

Define
\[
    h(u)
    =
    \frac{\alpha}{u}-u-\alpha+1,
    \qquad
    0<u<1.
\]
From the defining equation for $\kappa(\gamma)$, dividing by
$\kappa(\gamma)$ gives
\[
    h(\kappa(\gamma))
    =
    \frac{1}{\gamma}.
\]
Therefore,
\[
    h(\kappa_1)=\frac{1}{\gamma_1},
    \qquad
    h(\kappa_2)=\frac{1}{\gamma_2},
    \qquad
    h(\kappa_0)=\frac{1}{\gamma_0}.
\]
Since
\[
    \frac{1}{\gamma_0}
    =
    \frac{1}{\gamma_1}
    +
    \frac{1}{\gamma_2},
\]
we have
\[
    h(\kappa_0)
    =
    h(\kappa_1)+h(\kappa_2).
\]

Now write
\[
    x=\kappa_1,
    \qquad
    y=\kappa_2.
\]
The optimally weighted ensemble excess-risk contribution, divided by
$\tau^2$, is
\[
    q(x,y)
    =
    \frac{
    xy(1-xy)
    }{
    x+y-2xy
    }.
\]
Thus
\[
    \mathcal R_{\mathrm{excess}}
    =
    \frac{q(x,y)}{\kappa_0}.
\]
It is therefore enough to prove
\[
    q(x,y)>\kappa_0.
\]

Because
\[
    h'(u)
    =
    -\frac{\alpha}{u^2}-1<0,
\]
the function $h$ is strictly decreasing on $(0,1)$. Hence
\[
    q(x,y)>\kappa_0
\]
is equivalent to
\[
    h(q(x,y))<h(\kappa_0).
\]
Since
\[
    h(\kappa_0)=h(x)+h(y),
\]
it is enough to show
\[
    h(q(x,y))<h(x)+h(y).
\]

A direct algebraic calculation gives
\[
\begin{aligned}
    h(q(x,y))-h(x)-h(y)
    &=
    -
    \frac{
    (1-x)(1-y)
    \left[
        \alpha(x+y-2xy)
        +(1-xy)(x+y-xy)
    \right]
    }{
    (1-xy)(x+y-2xy)
    }.
\end{aligned}
\]
Now \(0<x,y<1\), so
\[
    1-x>0,
    \qquad
    1-y>0,
    \qquad
    1-xy>0,
\]
and
\[
    x+y-2xy
    =
    x(1-y)+y(1-x)>0.
\]
Also,
\[
    \alpha(x+y-2xy)>0
\]
because \(\alpha>0\), and
\[
    (1-xy)(x+y-xy)>0.
\]
Therefore the entire fraction is strictly positive, and the leading
minus sign implies
\[
    h(q(x,y))-h(x)-h(y)<0.
\]
Thus
\[
    h(q(x,y))<h(x)+h(y)=h(\kappa_0).
\]
Since \(h\) is strictly decreasing, this implies
\[
    q(x,y)>\kappa_0.
\]
Therefore
\[
    \frac{q(x,y)}{\kappa_0}>1,
\]
which is exactly
\[
    \mathcal R_{\mathrm{excess}}>1.
\]

%%%%%%%%%%%%%%%%%%%%%%%%%%%%%%%%%%%%%%%%%%%%%%%%%%%%%%%
%\newpage

%\newpage
%\input{checklist.tex}

\end{document}